\documentclass[12pt]{amsart}

\usepackage{amssymb,amsfonts,graphics,mathrsfs,color,tikz,enumitem}

\theoremstyle{plain}
\newtheorem{theorem}{Theorem}[section]
\newtheorem{corollary}[theorem]{Corollary}
\newtheorem{definition}[theorem]{Definition}
\newtheorem{lemma}[theorem]{Lemma}
\newtheorem{proposition}[theorem]{Proposition}

\newtheorem{remark}[theorem]{\it Remark}
\newtheorem{notation}[theorem]{\it Notation}

\newcommand\encircle[1]{\tikz[baseline=(X.base)] 
    \node(X)[draw, shape=circle, inner sep=-2]{\strut #1};}

\def\aa{a}
\def\E{E}
\def\cuv{\bfu\cd\bfv}
\def\JJ{\mathbb{J}}

\def\Sf{\mathbb{S}^2}
\def\Sm{\mathbb{S}^1}
\def\half{{\ts\frac12}}
\def\qart{{\ts\frac14}}
\def\pD{\partial D^3}
\def\oz{{\overline z}^{}}
\def\tz{{\tilde z}^{}}
\def\1{U(1)}
\def\2{G_2}
\def\3{SO(3)}
\def\LaS{\La^2_-T^*\!S^4}
\def\eu{e}
\def\gg{g^{}}
\def\hh{h^{}}
\def\BS{h^{}_\mathrm{BS}}
\def\hj{\hat\j}
\def\wg{{\widehat g}^{}}
\def\wh{{\widehat h}^{}}
\def\pq{Q}
\def\wF{\widehat F}
\def\bfq{\mathbf{q}}
\def\wfu{\mathbf{U}}
\def\wfv{\mathbf{V}}
\def\buv{(\bfu,\bfv)}
\def\smz{\setminus\mathbf0}
\def\vasphi{\hbox{$*\varphi$}}
\def\BSphi{\varphi^{}_\mathrm{BS}}
\def\BSsi{\sigma^{}_\mathrm{BS}}
\def\chern{\mathbf{c}}
\def\ch{\raise2pt\hbox{$\chi$}}
\renewcommand{\,}{\kern1pt}
\renewcommand{\!}{\kern-2pt}
\def\ceq{\kern4pt\hbox{\small\encircle{$=$}}\kern4pt}
\def\ceqq{\kern4pt\hbox{\small$\widehat{\encircle{$=$}}$}\kern4pt}
\def\T#1,#2,#3,{\{#1,#2,#3\}}
\def\qbox#1{\quad\hbox{#1}\quad}
\def\rbox#1{\raise4pt\hbox{$#1$}}
\def\lbox#1{\lower4pt\hbox{$#1$}}
\def\+{\!+\!}
\def\-{\!-\!}
\def\={\!=\!}
\def\cd{\kern-1pt\cdot\kern-1pt}
\def\ti{\!\times\!}
\def\ep{l}
\def\vep{\varepsilon}
\def\bfA{\mathbf{A}}
\def\bfB{\mathbf{B}}

\def\bfn{\mathbf{n}}
\def\bfp{\mathbf{p}}
\def\bfq{\mathbf{q}}
\def\bfu{\mathbf{u}}
\def\bfv{\mathbf{v}}
\def\bfs{\hbox{\boldmath$\sigma$}}
\def\bft{\hbox{\boldmath$\tau$}}
\def\bfm{\mathbf{m}}
\def\bfx{\mathbf{x}}
\def\bfz{\mathbf{z}}

\def\q{q^{}}
\def\u{u^{}}
\def\v{v^{}}
\def\x{x^{}}
\def\z{z^{}}
\def\gr{\color[rgb]{0,0,0}}
\def\ge{\geqslant}
\def\le{\leqslant}
\def\y{\\[3pt]}
\def\yy{\\[5pt]}
\def\yyy{\\[10pt]}
\def\ip{\raise1pt\hbox{\large$\lrcorner$}\>}
\def\suml{\sum\limits}
\def\vph{\vphantom{\lower2pt\hbox{p}\raise5pt\hbox{d}}}
\def\al{\alpha}
\def\tu{\tau^{}}

\def\th{\theta}
\def\Th{\Theta}
\def\hal{\hat\alpha}
\def\sC{\mathscr{C}}
\def\sF{\mathscr{F}}
\def\sL{\mathscr{L}^{}}
\def\sO{\mathscr{O}}
\def\sM{\mathscr{M}}
\def\sW{\mathscr{W}}
\def\Im{\mathop{\mathrm{Im}}}
\def\Re{\mathop{\mathrm{Re}}}
\def\Up{\Upsilon}
\def\pd{\partial^{}}
\def\Spin{\mathit{Spin}}
\def\de{\delta}
\def\ga{\gamma}
\def\la{\lambda}
\def\om{\omega}
\def\si{\sigma}

\def\Ga{\Gamma}
\def\La{\Lambda}

\def\op{\oplus}
\def\we{\wedge}
\def\ds{\displaystyle}
\def\ts{\textstyle}
\def\ba{\begin{array}}
\def\ea{\end{array}}
\def\be#1{\begin{equation}\label{#1}}
\def\ee{\end{equation}}
\def\bt{\begin{tabular}}
\def\et{\end{tabular}}
\def\ot{\!\otimes\!}

\def\lra{\longrightarrow}
\def\lmt{\longmapsto}
\def\ol{\overline}
\def\C{\mathbb{C}}
\def\CP{\mathbb{CP}}
\def\HP{\mathbb{HP}}
\def\H{\mathbb{H}}
\def\R{\mathbb{R}}
\def\Z{\mathbb{Z}}

\def\fT{\mathfrak{T}}
\def\fm{\mathfrak{m}}
\def\sp{\mathfrak{sp}}
\def\su{\mathfrak{su}}
\def\fu{\mathfrak{u}}
\def\vs{\vskip5pt}
\def\twenty#1{\fontsize{20}{20}\selectfont\bf #1}
\def\sm{\setminus}

\begin{document}

\parskip2pt
\parindent16pt
\mathsurround.5pt

\centerline{\twenty A circle quotient of a G$_2$ cone}

\vskip25pt

\centerline{Bobby Samir Acharya, Robert L.\ Bryant, and Simon Salamon}

\vskip25pt

\begin{quote}\small 
{\bf Abstract.} A study is made of $\R^6$ as a singular quotient of the conical
space $\R^+\times\CP^3$ with holonomy $\2,$ with respect to an obvious action
by $\1$ on $\CP^3$ with fixed points. Closed expressions are found for the
induced metric, and for both the curvature and symplectic 2-forms
characterizing the reduction. All these tensors are invariant by a diagonal
action of $\3$ on $\R^6,$ which can be used effectively to describe the
resulting geometrical features.
\end{quote}

\tableofcontents

\vspace{-20pt}

\section{Introduction}

Let $\sC$ denote the real 7-dimensional manifold $\R^+\times\CP^3$ endowed with
its $SO(5)$-invariant conical metric whose holonomy group is conjugate to
$\2$. We shall denote this Riemannian metric by $\hh_2$. Its isolated
singularity can be smoothed by passing to a complete $\2$ metric on the total
space of half the bundle of 2-forms over $S^4$ (as two of the authors showed in
\cite{BS} and others in \cite{GPP}). Restricting to unit 2-forms defines the
Penrose twistor fibration
\[\pi\colon\ \CP^3\lra S^4,\]
which also plays a key role in understanding the conical $\2$ structure. The
latter is most easily defined by means of a closed 3-form $\varphi,$ which
determines $\hh_2,$ and a closed 4-form $\vasphi$ (with $*$ defined by $\hh_2$
and ultimately $\varphi$).

We shall study a quotient \be{Q} \pq\colon\ \sC\lra\sM\ee of $\sC$ by a natural
circle subgroup $\1$ with fixed points, and identify the metric $\gg_2$ induced
on the 6-dimensional base $\sM$. The group $\1$ is induced by left
multiplication by $e^{i\th}$ on $\H^2=\C^4,$ commuting with the
projectivization to $\CP^3$. The same action can be obtained from rotation of
two coordinates of $\R^5=\R^2\op\R^3,$ since this lifts to the one we want via
$\pi$. Modulo the origin, $\sM$ is a $\1\times \1$ quotient of $\R^8$ and,
applying the Gibbons-Hawking ansatz \cite{GH}, one can identify $\sM$ with
$\R^3\times\R^3$ in which each `axis' $\R^3$ arises from the fixed point set of
the respective $\1$.

A point of $\sM$ will be represented by a `bivector' $\buv$ with $\bfu\in\R^3$
and $\bfv\in\R^3$. (Our terminology acknowledges that of \cite{WRH}, in which a
bivector is the vector part of a quaternion with complex coefficients, in our
case a hyperk\"ahler moment map.)  This description is closely related to the
K\"ahler quotient
\[\CP^3/\!/\,\1\>\cong\>\CP^1\times\CP^1.\]
The K\"ahler picture, and an associated metric $\gg_1$ on $\sM,$ provides a
useful comparison for some of our results. However, we are primarily interested
in tensors arising from $\2,$ which explains why we add an action by $\R^+$
rather than remove one. {\gr Inside $\CP^3,$ the circle action fixes two
  projective lines, the twistor lifts of the fixed 2-sphere
  $\Sf=S^4\cap\R^3$. The two projective lines are swapped by an anti-linear
  involution $j$ (see Lemma \ref{j}) and lie in a unique $U(2)$-invariant
  complex quadric in $\CP^3$ (defined in Section \ref{SO3}).}

Each $\1$ (or more effectively, $SO(2)$) orbit on $S^4$ is specified by a
unique point of norm at most one in $\R^3,$ so we can identify $S^4/SO(2)$ with
the closed unit ball $D^3$. This gives rise to a commutative diagram in which
$\varpi$ is induced by $\pi$ and an $\R^+$ quotient:\vs

\[\ba{ccc} 
&\H^2\smz&\\[5pt]
\rbox{\rho}\swarrow&\hskip30pt&\searrow\rbox{\ \ }\\[10pt]
\sC=\R^+\ti\CP^3\hskip30pt&
\stackrel{\hbox{$\pq$}}\lra&\hskip20pt\R^6\smz=\sM\\[15pt]
\lbox{\ }\searrow&\hskip50pt&\swarrow\lbox{\varpi}\yyy
& D^3&
\ea\]\par
\[\hbox{Figure 1: Quotients described by spaces of dimension 
$7,$ $6$ and $3$}\]\vs\vs

{\gr Significant motivation for this work came from physics via Atiyah and
  Witten who first highlighted many aspects of the rich geometry in this
  construction in the context of M-theory and the relation to Type IIA
  superstring theory \cite{AW}. The idea is that, physically, M-theory
  formulated on $\sC$ is dual to Type IIA superstring theory on $\sM$. The
  fixed points of the $U(1)$ action on the $\2$ manifold are identified with
  $D$6-branes in Type IIA theory on $\sM$. In \cite{AW}, much of the discussion
  concerning this point is topological in nature and does not focus on the
  metric on $\sM,$ which is a primary focus of this work.}

By analogy to the families of metrics with $\2$ holonomy interpolating between
highly-collapsed metrics and those asymptotic to the cone over $S^3\ti S^3$ in
\cite{FHN}, one might expect $\sM$ to acquire a Calabi-Yau metric and the
singular $\R^3$'s to be special Lagrangian. In our situation, there is no such
collapsed limit because $\sC$ has no finite circles at infinity, and our work
shows that the picture painted in \cite{AW} is somewhat of an
oversimplification.  However, we do show that the induced symplectic form $\si$
is very easy to describe on $\sM$ and that the singular $\R^3$'s are
Lagrangian, {\gr as foreseen in \cite{AW}.}

Our aim is to describe the $SU(3)$ structure $(\gg_2,\JJ,\si)$ induced on the
smooth locus $\sM'$ of $\sM$. Contrary to the assumption adopted in
\cite{ApoS}, the quotient is not K\"ahler, but one can rescale $\gg_2$ so that
$\si$ has constant norm and we are dealing with an \emph{almost K\"ahler
  structure}. There is a residual diagonal action of $\3$ on $\sM'$ that
preserves the tensors $\gg_2,\JJ,\si$. The singular nature of the quotient
makes our initial formulae complicated, as they involve radii functions that
are not smooth across $\R^3\cup\R^3$. Part of our task is to find coordinates
on $\R^6,$ or subvarieties thereof, that are better adapted to $\si$ and
$\gg_2$.

\vskip20pt\label{gloss}

\begin{center}
\bt{|l|c|l|}\hline
\multicolumn{3}{|c|}{\sc Glossary of notation}\\\hline\hline\vph
tensor & defined on/in & description\\\hline\hline\vph
$\eu$ & $\R^8$ & Euclidean metric\\ 
$R$ && Euclidean norm squared\\
$X$ && Killing vector field\\\hline\vph
$\wh_1$ & $\CP^3$ & K\"ahler metric\\
$\wh_2$ && nearly-K\"ahler metric\\
$\om$ && nearly-K\"ahler 2-form\\
$\Up$ && nearly-K\"ahler $(3,0)$-form\\\hline\vph
$\hh_2$ & $\sC$ & $\2$ metric\\
$\hh_c$ && more general conical metric\\
$\BS$ && complete $\2$ metric\\
$\varphi$ && $\2$ 3-form\\
$\vasphi$ && $\2$ 4-form\\
$\Th_c$ && connection 1-form\\\hline\vph
$g_c$ & $\sM$ & metric induced from $h_c$\\
$\wg_c$ && restriction of $\gg_c$ to $R=1$\\
$F_c=d\Th_c$ && curvature 2-form\\
$\si$ && symplectic 2-form\\
$\JJ$ && almost complex structure\\
$\Psi=\psi^++i\psi^-$ && $(3,0)$-form\\
$\sF_+,\sF_-$ && $\3$-invariant subvarieties\\
$\sM(\bfn)$ && $\JJ$-holomorphic subvarieties\\\hline
\et
\end{center}

\vskip15pt

We introduce the Gibbons-Hawking ansatz in Section \ref{GH}, and apply it to a
baby model of a circle quotient of a $\2$ structure. To analyse our curved
example, we pull $\hh_2$ and other tensors back (via $\rho$) to $\H^2\cong\R^8$
in Section \ref{holo}, and exploit the ambient hyperk\"ahler structure. {\gr The
novelty in our approach consists of expressing everything in Euclidean terms on
$\R^8.$} In Section \ref{T2}, we consider commuting circle actions and describe
the $\2$ structure on $\sC$. This is motivated by work of the first author
\cite{AchW}, and we hope to use our methods subsequently to understand circle
actions with different weights on $\R^8$. Using the 2-torus action on $\R^8$
and the map $\pq,$ we identify the metric $\gg_1$ on $\sM$ arising from the
Fubini-Study metric of $\CP^3,$ and the more complicated metric $\gg_2$ induced
from the $\2$ structure of $\sC$ (Theorems \ref{g1} and \ref{g2}).

The diagonal action of $\3$ on $\sM\subset\R^6$ enables us to use the
bivector formalism to describe invariant tensors in Section
\ref{curv}. We show that the curvature 2-forms $F_1,F_2,$ are
determined by their restrictions to $S^2\ti S^2$ (Theorem \ref{F}).
In Section \ref{SO3}, we focus on two 4-dimensional $\3$-invariant
submanifolds $\sF_\pm$ in $\sM$ such that $\sF_+$ projects to $\pd
D^3,$ while the circle fibres of $\pq$ are horizontal over $\sF_-$. As
an application, we use the twistor fibration to describe a foliation
of $\sC$ by coassociative submanifolds discovered by Karigiannis and
Lotay \cite{KL2} (Theorem \ref{cas}).

In Section \ref{SU3}, we show that Darboux coordinates for $\si$ can
be expressed remarkably simply in terms of the bivector $\buv,$ though
this result (Theorem \ref{si}) was by no means obvious. It contrasts
with the difficulty in describing the almost complex structure $\JJ,$
though we compute a compatible $(3,0)$-form on $\sM$ and verify that
$\JJ$ is non-integrable. The nature of this 3-form leads us to exhibit
a family of 4-dimensional pseudo-holomorphic linear subvarieties
$\sM(\bfn)$ parametrized by $\mathbb{RP}^2$ that exhaust $\sM$.

We investigate the metrics $\gg_1$ and $\gg_2$ in Section \ref{met},
and distinguish subvarieties on which they are flat. In particular, we
determine their restriction to $\sF_+$ and $\sF_-$
(Theorem \ref{sF+-}), and highlight geometrical aspects that `ignore'
the singularities of $\sM$.\vs

\noindent{\small\textbf{Acknowledgments.} The authors are supported by
  the Simons Collaboration on Special Holonomy in Geometry, Analysis,
  and Physics (\#488569 Bobby Acharya, \#347349 Robert Bryant,
  \#488635 Simon Salamon).} They are grateful to Spiro Karigiannis and
Jason Lotay for sharing some results from \cite{KL2} discussed in
Section \ref{SO3}. The third author acknowledges useful exchanges of
ideas with Benjamin Aslan, Udhav Fowdar, Wendelin Lutz, and Corvin
Paul.\vs\vs

\setcounter{equation}0
\section{Preliminaries}\label{GH}

This section serves both to motivate the more technical work that
follows, and to introduce the Gibbons-Hawking ansatz in a simple $\2$
context.

In the study of Ricci-flat metrics with special holonomy, there is a
well-established connection between structures defined by the Lie
groups $SU(3),$ $\2$ and $\Spin(7)$ in dimensions 6, 7 and 8.  Hand in
hand with the condition of \emph{reduced holonomy} is that of
\emph{weak holonomy}; the former is characterized by the existence of
a non-zero \emph{parallel} spinor, the latter by a \emph{Killing}
spinor. If an $n$-dimensional manifold $M^n$ (with $n=5,6,7$) has a
Riemannian metric $g$ with weak holonomy, then the cone $dr^2+r^2g$
has reduced holonomy on $M^n\times\R^+$ \cite{Bar}, and the sine-cone
$dr^2+(\sin r)^2g$ has weak holonomy on $M^n\times(0,\pi)$ \cite{AL}.
Examples of such metrics permeate this paper, though our focus will be
on quotienting a 7-dimensional manifold by a circle action.

In this paper, we restrict attention to $\2$ holonomy in seven
dimensions, and $SU(3)$ structures (invariably without reduced
holonomy) in six dimensions. A $\2$ structure on a 7-manifold $M$ is
determined by a `positive' non-degenerate 3-form $\varphi$ that
satisfies
\[ d\varphi=0\qbox{and} d\vasphi=0.\] 
Here $*$ is Hodge star for the Riemannian metric $\hh_2$ uniquely
determined by (i) the formula
\[\ts \hh_2(X,Y)\upsilon =
\frac16(X\ip\varphi)\we(Y\ip\varphi)\we\varphi,\]
and (ii) the condition that $\upsilon$ be the volume form of $\hh_2$
with an appropriate orientation choice \cite{Br1}. It is then the case
that $\nabla\varphi=0,$ where $\nabla$ is the Levi-Civita connection
for $\hh_2$.

An analogous description of $SU(3)$ (`Calabi-Yau') holonomy consists
of a symplectic 2-form $\si$ and a complex closed 3-form
$\Psi=\psi^++i\psi^-$ satisfying $\Psi\we\si=0$ and
\[\ts -i\,\Psi\we\ol\Psi = \frac43\,\si^3\]
in the notation of \cite{ChS}. The real and imaginary components of
$\Psi$ must be \emph{stable} in the sense that the stabilizer of
either in $GL(6,\R)$ is conjugate to $SL(3,\C),$ in order that
$\psi^+$ determine an almost complex structure $J$ and
$\psi^-=J\psi^+$ \cite{Hit}. The 2-form $\si$ will necessarily have
type $(1,1)$ relative to $J,$ but we also require that the
non-degenerate bilinear form $\gg_2=\si(J\cdot,\cdot)$ be positive
definite.

More general $SU(3)$ structures are defined by merely relaxing the
closure conditions on $\si$ and/or $\psi^\pm$. When an $SU(3)$
structure arises as a hypersurface of a manifold with holonomy $\2,$
its torsion $\fT$ `loses' half of its 42 components. For such an
embedding $M^6\hookrightarrow M^7,$ one defines
\[ \left\{\>\ba{rcl} 
\psi^+ &=& i^*\varphi,\y \frac12\si^2 &=& i^*(\vasphi),\ea\right.\]
and it is the closure of these two differential forms {\gr that is} the
half-flat condition.

Now consider the quotient situation. Suppose $(M^7,\varphi)$ has
holonomy in $\2$ and that $\1$ acts freely on $M^7$ with associated
Killing vector field $X$. Then $\sL_X\varphi=0$ and
\[ \si=X\ip\varphi\]
is closed. Let $\ep$ be the positive function defined by
\be{s} \ep^{-4}=\hh_2(X,X),\ee
so that $\ep^{-2}=\|X\|$ measures the size of the $\1$ fibres. Let
$\th=\ep^4X\ip\hh_2$ so that $X\ip\th=1$. Following \cite{ApoS} (where
$t$ corresponds to $\ep^2$ and the signs of $\psi^\pm$ are swapped),
one can write
\[\ba{rcl} 
\varphi &=& \th\we\si + \ep^3\psi^-,\y 
\vasphi &=& \th\we(\ep\,\psi^+) + \frac12(\ep^2\,\sigma)^2. 
\ea\]
The 1-form $i\th$ defines a connection on the $\1$ bundle, and
$F=d\th$ equals ($-i$ times) its curvature. The latter is constrained
by the residual torsion:

\begin{lemma}\label{FF}
The differential forms $\si$ and $\Psi=\psi^++i\psi^-$ define an
$SU(3)$ structure on $M^7/\1$ with $d\si=0$ and $d(\ep\psi^+)=0$.
Moreover,
\[ \ba{rcl} 
F\we\si    &=& -d(\ep^3\psi^-)\y
F\we\psi^+ &=& -2\ep^2\,d\ep\we\si^2.\ea\] 
\end{lemma}

\begin{proof}
The required algebraic properties of the exterior forms follow from
the well-known linear algebra linking $SU(3)$ and $\2$ structures, so
we confine ourselves to understanding the exponents of $\ep$ in the
expressions for $\varphi$ and $\vasphi$ above the lemma. We want the
four terms to have constant norm relative to $\hh_2$ on $M^7$. This
implies that $\ep^2\si$ should have constant norm, since $\|\ep^{-2}\th\|=1,$
and the 3-form $\psi^-$ is scaled by $(\ep^2)^{3/2}=l^3$ for
consistency. For the same reason, $(\ep^{-2}\th)(\ep^3\psi^+)$ has
constant norm.

Since $X\ip\vasphi=\ep\,\psi^+,$ the latter is indeed closed. The
equations involving $F$ follow immediately by differentiation.
\end{proof}

\begin{remark}\label{scale}\rm
One is free to scale the metric induced on $M^7/\1$ by any function of
$\ep,$ and the property $d\si=0$ characterizes the choice of an
\emph{almost K\"ahler} metric. However, the metric $\gg_2$ for which
\[(M^7,\hh_2)\lra(M^7/\1,\gg_2)\] is a Riemannian submersion
corresponds to the re-scaled $SU(3)$ structure $(\ep^2\si,\>\ep^3\Psi)$. The
almost complex structure $J$ is integrable if and only if $d(\ep\psi^-)=0,$ and
in this case it was shown in \cite{ApoS} that (i) the Ricci form of the
K\"ahler metric equals $\gr i\pd\ol\pd\log(l^2),$ and (ii) a new Killing vector
field $U$ is defined by $U\ip\sigma=-d(\ep^2),$ and one can further quotient to
4 dimensions.  Other reductions leading to triples of 2-forms and
Monge-Amp\`ere equations can be imposed with extra symmetry \cite{Don}.
\end{remark}

We shall rely repeatedly on the first part of Lemma \ref{FF} in the
sequel, though the function $N=\ep^{-4}$ will be more relevant
computationally, and we shall only use the symbol $\ep$ in this
section. We conclude it by applying the theory above to a $\1$
quotient of the flat $\2$ structure on $\R^7,$ specified by means of
the constant 3-form \be{c3f}
\varphi=dx_{014}-dx_{234}+dx_{025}-dx_{315}+dx_{036}-dx_{126}+dx_{456},\ee
using coordinates $\x_i$ with $0\le i\le 6$. This defines an inclusion
$\2\subset SO(7),$ for which the orthogonal group fixes the Euclidean
metric $\eu=\sum_{i=0}^6\!dx_i^2$. We further distinguish the subspace
$\R^4=\R^4_{0123},$ and consider the action of $\1$ on this subspace
giving rise to the Killing vector field
\[ X = -\x_1\pd_0+\x_0\pd_1-\x_3\pd_2+\x_2\pd_3,\]
where $\pd_0=\pd/\pd x_0$ etc. There is an associated 1-form
\[\xi = X^\flat = X\ip e 
      = -\x_1d\x_0+\x_0d\x_1-\x_3d\x_2+\x_2d\x_3,\]
and
\[\ts \u_0 = X\ip \xi = \suml_{i=0}^3 x_i{}\!^2\] 
is the norm squared of both $X$ and $\xi$. If we set $\th=\xi/u_0,$
then $i\th$ is a connection form for the smooth circle bundle over
$\R^4\smz$.

We next identify $\R^4\cong\H$ by means of the quaternionic coordinate
\[ q=\x_0\+\x_1i\+\x_2j\+\x_3k.\] 
A hyperk\"ahler structure on $\R^4_{0123},$ specified by the anti-self-dual
(ASD) 2-forms
\be{ASD} d\x_{01}-d\x_{23},\quad d\x_{02}-d\x_{31},\quad d\x_{03}-d\x_{12},\ee
where $d\x_{ij}$ is shorthand for $\x_i\we d\x_j$. The action of $\1$ is
triholomorphic. The associated moment mapping is
\be{assoc}\gr q\ \lmt\ -\half\ol q\,i\,q = -\half(\u_1i-\u_3j+\u_2k),\ee
where (to suit the authors' conventions, cf.\ \eqref{UV})
\[\left\{\ \ba{rcl}
\u_1 &=& x_0^2 + x_1^2 - x_2^2 - x_3^2\y
\u_2 &=& 2(\x_0\x_2 + \x_1\x_3)\y
\u_3 &=& 2(\x_0\x_3 - \x_1\x_2).
\ea\right.\]
The map \eqref{assoc} is invariant by the $\1$ action $q\mapsto e^{i\theta}q,$
and defines a homeomorphism $\R^4/\1\cong\R^3$.

The curvature can be expressed in terms of the basis \eqref{ASD} and the
self-dual curvature form $d\xi=2(d\x_{01}+d\x_{23})$:
\[ d\th = \frac1{u_0}d\xi-\frac1{u_0^2}(\u_1\om_1-\u_3\om_2+\u_2\om_3).\]
The $\u_i$ provide a smooth structure on $\R^4,$ and we can express
$\1$ invariant quantities in terms of these coordinates. In
particular,
\[ u_0^2 = u_1^2+u_2^2+u_3^2,\]
so that $\u_0$ is the radius and $u_0^{-1}$ is harmonic in the $\u_i$
coordinates. The Euclidean metric on $\R^4$ can then be recovered by
means of the Gibbons-Hawking ansatz; it is
\[ \u_0\,\th^2 + \qart u_0^{-1}\sum_{i=1}^3\!du_i^2.\]
The second summand on the right equals the metric induced submersively
on the quotient $\R^3$. The factor of $1/4$ (and $1/2$ in the
lemma below) could be eliminated by halving the coordinates $\u_i,$
but that would be inconvenient later.

\begin{lemma}\label{star}
\[\ts d\th = \frac12u_0^{-3}\big(\u_1\,d\u_2\we d\u_3
     +\u_2\,d\u_3\we d\u_1+\u_3\,d\u_1\we d\u_2\big).\]
\end{lemma}

This result is well known (up to the factor of $1/2$), since the
right-hand side equals
\[ \half u_0^{-3}\,*\!(\u_1d\u_1+\u_2d\u_2+\u_3d\u_3) = 
-\half\,*d(u_0^{-1}).\] In the sequel, we shall denote this expression
by $\qart u_0^{-3}\T\bfu,d\bfu,d\bfu,$ (an extra factor of 2 converts
the $1/4$ into $1/2,$ see Notation \ref{triple}). Such triple
products will be used to express various tensors. If we restrict to
the 2-sphere $\u_0=1$ and adopt spherical coordinates $\phi$
(latitude) and $\th$ (longitude), then
\[ -2\,d\th = \cos\phi\ d\th\we d\phi\] 
is the area 2-form, whose integral equals $4\pi$. It follows that the
circle bundle has first Chern class $\chern_1=-1$ over $S^2$; it is
the Hopf bundle. For contrasting applications of the Gibbons-Hawking
ansatz in four dimensions, see \cite{LeB,deB}.

Thus far, we have dealt only with 4-dimensional geometry. Given that
$\1$ acts trivially on $\R^3_{456},$ the quotient of $\R^7$ is
\[ \frac{\R^4}{\1}\times\R^3\ \cong\ \R^6,\] 
with coordinates $(\u_1,\u_2,\u_3;\x_4,\x_5,\x_6)$. We we can now
identify the structure induced from the 3-form \eqref{c3f}:

\begin{proposition}\label{flat} 
The quotient $\R^6$ has an induced $SU(3)$ structure with $\u_0=\ep^{-4},$
\[\ba{rcl}
\si  &=& -\frac12(d\u_1\we d\x_4-d\u_3\we d\x_5+d\u_2\we d\x_6),\yy
\Psi &=& -(\frac12\ep\,d\u_1+i\ep^{-1}d\x_4)\we
(-\frac12\ep\,d\u_3+i\ep^{-1}d\x_5)\we(\frac12\ep\,d\u_2+i\ep^{-1}d\x_6).\ea\]
\end{proposition}

\begin{proof}
The $SU(3)$ structure is completely determined by $\si$ and $\psi^+$. The first
equation follows from the definition $\si=X\ip\varphi$. {\gr In accordance with
  the way the coordinates $\u_i$ are ordered in \eqref{assoc}, we also have}
\[\ba{rl} \ep\psi^+ \kern-6pt
&= X\ip\vasphi\yy 
&= \frac12(d\u_1\we d\x_{56}-d\u_3\we d\x_{64}+d\u_2\we d\x_{45})
-\x_0d\x_{023}-\x_1d\x_{123}-\x_2d\x_{012}-\x_3d\x_{013}\yy 
&= \frac12(d\u_1\we d\x_{56} - 
d\u_3\we d\x_{64}+d\u_2\we d\x_{45})-\frac18\ep^4\,d\u_{123}.\ea\] 
The last line equals $l$ times the real part of the simple 3-form
$\Psi$ specified by the proposition. Since the stable form $\psi^+$
determines $J,$ it follows that $\Psi=\psi^++i\psi^-$ is a $(3,0)$
form compatible with $\si$.
\end{proof}

Although the starting metric $\eu$ is flat, the circle bundle is not,
and the torsion $\fT$ of the $SU(3)$ structure above is determined by
$d(\ep\psi^-)\ne0$. This confirms that the quotient is not K\"ahler, but
provides results that are entirely consistent with Lemma \ref{FF}. We
purposely chose a circle subgroup that acts trivially on $\R^3,$
though other $SU(3)$ structures can be defined by a different choice
of
\[ \1\subset SO(4)\subset\2\]
acting on $\R^7\cong\R^4\op\La^2_-(\R^4)$. Proposition \ref{flat}
exhibits the simplest model for quotients of metrics with holonomy
$\2,$ including the one on $\sC$ that is the main focus of this
paper. The equations to set up the quotient are identical, but we
shall inevitably struggle to find such simple expressions for the
induced differential forms in a non-flat situation.

\setcounter{equation}0
\section{Metrics with holonomy $\2$}\label{holo}

At this point, we need to refresh notation, and establish our choice
of real, complex and quaternionic coordinates on
\[ \R^8=\C^4=\H^2\]
that will persist for the remainder of the paper.  We shall consider
$\R^8$ as a module for the Lie group $Sp(2)Sp(1),$ with the group
$Sp(2)$ of quaternionic matrices acting by \emph{left} multiplication,
and the group $Sp(1)$ of unit quaternion scalars acting on the
\emph{right}. We set
\be{qq}\ba{lll}
\q_0 &=\ \x_0+\x_1i+\x_2j+\x_3k &=\ \z_0+j\z_1,\y
\q_1 &=\ \x_4+\x_5i+\x_6j+\x_7k &=\ \z_2+j\z_3,
\ea\ee
so that
\be{zzzz}
\z_0=\x_0+i\x_1,\quad \z_1=\x_2-i\x_3,\quad 
\z_2=\x_4+i\x_5,\quad \z_3=\x_6-i\x_7.\ee 
Then $\z_i$ and $\q_j$ become homogeneous coordinates for $\CP^3$ and
$\HP^1$ respectively, consistent with the choice of \emph{right}
multiplication by $\H^*$.

The twistor projection $\pi\colon\CP^3\to\HP^1$ is represented by
\[[\z_0,\z_1,\z_2,\z_3]\mapsto[\q_0,\q_1]=[1,\>\q_1q_0^{-1}],\]
in which the point at infinity is defined by $\q_0=0$. 
Away from this point,
\be{pi}\ba{rcl}\ds \q_1q_0^{-1}\ =\ \frac1{|q_0|^2}\q_1\ol{\q_0}
&=&\ds \frac1{|z_0|^2+|z_1|^2}(\z_2+j\z_3)(\oz_0-j\z_1)\yy &=&\ds
\frac1{|z_0|^2+|z_1|^2}\Big( \z_2\oz_0+\oz_3\z_1 + \gr j(\oz_0\z_3-\z_1\oz_2)\Big).
\ea\ee
This convention will determine the chirality of the Killing vector fields
defined below.

We denote by
\[\ts R=\suml_{i=0}^7 x_i^2\]
the radius squared for the Euclidean metric
\[\ts \eu=\suml_{i=0}^7 dx_i\ot dx_i.\]
The vector fields $\pd_0\=\pd/\pd x_1,\ldots,\pd_7\=\pd/\pd x_7$ constitute an 
orthonormal basis for $e$.

The right action of $Sp(1)$ determines Killing vector fields
\[\ba{l} 
Y_1 = -\x_1\pd_0+\x_0\pd_1+\x_3\pd_2-\x_2\pd_3-\x_5\pd_4+\x_4\pd_5+\x_7\pd_6-\x_6\pd_7\y
Y_2 = -\x_2\pd_0+\x_0\pd_2+\x_1\pd_3-\x_3\pd_1-\x_6\pd_4+\x_4\pd_6+\x_5\pd_7-\x_7\pd_5\y
Y_3 = -\x_3\pd_0+\x_0\pd_3+\x_2\pd_1-\x_1\pd_2-\x_7\pd_4+\x_4\pd_7+\x_6\pd_5-\x_5\pd_6,
\ea\]
tangent to the fibres of the principal bundle
\[ S^7=\frac{Sp(2)}{Sp(1)}\lra\frac{Sp(2)}{Sp(1)\ti Sp(1)}=S^4.\] 
With the conventions below, we shall identify the associated rank 3
vector bundle with the bundle $\LaS$ of \emph{anti-self-dual} 2-forms.

Consider the 1-forms $\al_i=Y_i\ip g,$ such as
\[ \al_1 = -\x_1d\x_0+\x_0d\x_1+\x_3d\x_2-\x_2d\x_3
-\x_5d\x_4+\x_4d\x_5+\x_7d\x_6-\x_6d\x_7.\]
Observe that 
\[ \half d\al_1 = d\x_{01}-d\x_{23}+d\x_{45}-d\x_{67}\]
is the sum of two ASD 2-forms on separate $\R^4$'s. Now consider the
action by the group $\R^+$ of positive real scalars on $\R^8,$ so that
$\x_i\mapsto\la\x_i$ for $\la\in\R^+$. The 1-forms
\[ \hal_i = \frac{\al_i}R\] are invariant by this action, and (by
interpreting $Y_i$ as elements of $\mathfrak{so}(3)$)
\[ \vartheta=\sum_{i=1}^3\hal_i\ot Y_i\] is a
connection form on the principal $\H^*$ bundle. Its curvature equals
\[\ba{rcl}
 d\vartheta+[\vartheta,\vartheta]
&=& (d\hal_1+2\hal_2\we\hal_3)\ot Y_i +\cdots\yy
&=& \suml_{i=1}^3 \tu_i\ot Y_i,\ea\]
where
\[\left\{\ba{rcl}
\tu_1 &=& d\hal_1+2\hal_{23}\y
\tu_2 &=& d\hal_2+2\hal_{31}\y
\tu_3 &=& d\hal_3+2\hal_{12}\ea\right.\] 
We can check that the coefficients and signs are correct by verifying
that
\[ Y_i\ip\tu_j=0\qquad i,j=1,2,3,\] 
which follows from equations such as $Y_i\ip\hal_i=1$ and $Y_1\ip
d\hal_2=2\hal_3$. This implies that the $\tu_i$ are semi-basic over
$S^4$. Bearing in mind that
\[\tu_i\we\tu_i\we\al_{123}\we dR = -16\,d\x_{01234567},\]
we shall {\gr choose} orientations on $S^7$ and $S^4$ so that
$\{\tu_1,\tu_2,\tu_3\}$ a basis of \emph{anti}-self-dual 2-forms over $S^4$.

Fix $i=1$ and consider the subgroup $\1_1$ generated by $Y_1$. This
will be our choice for defining the complex projective space
\[\CP^3 = \frac{S^7}{\1_1}.\]
The 1-form $\hal_1$ determines a connection on the $\1_1$ bundle
$S^7\to \CP^3,$ with curvature 2-form proportional to $d\hal_1$. The
rescaling of $\al_1$ ensures that
\[ Y_1\ip d\hal_1 = \sL_{Y_1}\hal_1-d(Y_1\ip\hal_1)=0,\] 
so that $d\hal_1$ passes to $\CP^3$. It is well known that $d\hal_1$
is the K\"ahler form for (a suitably normalized) Fubini-Study metric
on $\CP^3$. We denote the standard integrable complex structure on
$\CP^3$ by $J_1$.

The nearly-K\"ahler structure of $\CP^3$ is compatible with the non-integrable
almost complex structure $J_2,$ obtained from $J_1$ by reversing sign on the
twistor fibres \cite{ES}. We shall denote the $SU(3)$ structure on
$(\CP^3,J_2)$ by a 2-form $\om$ and a $(3,0)$-form $\Up,$ neither of which are
closed. (We shall reserve $\si$ and $\Psi$ to describe the $SU(3)$ structure of
{\gr the quotient $\R^+\times\CP^3$ by $SO(2),$} see the Glossary of Notation
on page \pageref{gloss}.)

\begin{lemma}\label{NK}
\[\ba{rcl} 
\om &=& \hal_{23}+\tu_1\ =\ d\hal_1+3\hal_{23}\yy
\Up &=& (\hal_2+i\hal_3)\we(\tu_2-i\tu_3)\\[-5pt]
\ea\]
\end{lemma}

\begin{proof}
Since
\[ d\hal_1=\tu_1-2\hal_{23}\]
is K\"ahler form on $(\CP^3,J_1),$ we must have
$\om=\tu_1+\la\hal_{23}$ for some $\la>0$. Using the fact that
$\tau_i^2$ defines the same volume form on $S^4$ independent of $i,$
one verifies that the nearly-K\"ahler identities, namely
\[ d\om = 3\,\Im\Up \qbox{and} d\Up = 2\,\om^2,\] 
hold for $\la=1$. The formula for the $(3,0)$-form $\Up$ follows from
the fact that $J_2=-J_1$ when restricted to the fibres over $S^4$.
\end{proof}

If a 6-manifold $M$ carries a nearly-K\"ahler metric then the conical
metric on $\R^+\times M$ has holonomy contained in $\2$. In
particular, when $M$ is the twistor space $(\CP^3,J_2),$ with isometry
group $SO(5),$ the conical metric has holonomy \emph{equal} to $\2$
\cite{BS}. We choose as radial parameter the Euclidean norm squared
\[ R=\|\bfq\|^2=|\q_0|^2+|\q_1|^2;\]
this will ensure that $\pq\circ\rho$ arises from a quadratic map in
the commutative Figure 1.

\begin{proposition}\label{exact} 
The conical $\2$ structure on $\sC$ is characterized by the exact
forms
\[\ts\varphi = d(\frac13R^3\om)\qbox{and} 
\vasphi = d(\qart R^4\,\Re\Up).\]
\end{proposition}

\begin{proof} 
Each 1-form on $\CP^3$ is weighted with the radial parameter $R$ with
respect to the conical metric, so
\[\ba{rcl}
\varphi &=& dR\we R^2\om+R^3\,\Im\Up,\yy
\vasphi &=& dR\we R^3\,\Re\Up+\frac12(R^2\om)^2.
\ea\]
The formulae follow.
\end{proof}

Let $\wh_1$ denote the K\"ahler metric of $(\CP^3,J_1)$ corresponding
to the 2-form $d\hal_1$. Let $\wh_2$ denote the nearly-K\"ahler metric
of $(\CP^3,J_2)$ determined by Lemma \ref{NK}. The next result
describes the pullbacks of the conical metrics
\be{whc} \hh_c = dR^2 + R^2\,\wh_c,\qquad c=1,2,\ee
to $\R^8,$ in term of the Euclidean metric $e$ and the 1-forms $\al_i$
defined previously. We are mainly interested in $\hh_2$ since this has
holonomy $\2,$ but related metrics will be useful for comparison
purposes.

\begin{proposition}\label{hc} 
  The metrics $\hh_1$ and $\hh_2$ on $\sC$ belong to the one-parameter
  family of bilinear forms
\be{bilinear}
\hh_c = \half dR^2+ 2R\,e-2\al_1^2+(1-c)\big(\al_2^2+\al_3^2\big).
\ee
These forms define Riemannian metrics on $\sC$ provided $c<3$.
\end{proposition}

\begin{proof}
The Euclidean metric is given by
\[ e = dr^2 + r^2s_7,\]
where $R=r^2,$ 
\[\ts 2s_7 = \pi^*\!s_4 + 2\suml_{i=1}^3\hal_i^2,\]
and $s_7,s_4$ are standard metrics on $S^7,S^4$. Starting from $e,$ we
can define $s_7$ and $s_4$ by these formulae and verify that $Y_i\ip
s_4=0$ for $i=1,2,3$. We have inserted a `$2$' in the definition
of $s_7$ to match the K\"ahler metric 
\[ 2(s_7 - \hal_1^2) = \pi^*\!s_4 + 2(\hal_2^2+\hal_3^2)\] 
on $\CP^3$ corresponding to the 2-form $d\hal_1= \tu_1+2\hal_{23}$ 
used in the proof of Lemma \ref{NK}.

The bilinear from defined by the proposition can now be expressed 
as
\be{3-c} \hh_c =
dR^2 + R^2\pi^*\!s_4 + (3-c)\big(\al_2^2+\al_3^2\big).\ee
When $c=2,$ this matches the metric
$\pi^*\!s_4 + \hal_2^2+\hal_3^2$
inherent in Lemma \ref{NK}. The bilinear form $\hh_c$ contains
$Y_1$ in its kernel since $\al_1=Y_1\ip\eu$ and $Y_1$ has zero
contraction with $dR,\al_2,\al_3$. Note that
\be{hh3}\hh_3 = dR^2 + R^2\,\pi^*\!s_4,\ee
but if $c<3$ then $\hh_c$ has rank 7 and is positive definite on the
quotient $\R^8/\left<Y_1\right>$.
\end{proof}

\begin{remark}\rm
The underlying family of metrics on $\CP^3$ described in the proof is
well known: it arises from Riemannian submersions $\CP^3\to S^4$ by
varying the scaling on the fibres. The cases $c=1$ and $c=2$
correspond to the two $Sp(2)$-invariant Einstein metrics on $\CP^3$
\cite{Besse,WZ}. In the former case, the conical metric can be
simply expressed as
\[ \hh_1 = dR^2 + 2\,R^2(s_7 - \hal_1^2),\]
and has six equal eigenvalues. {\gr Our approach (emphasized in the
  Introduction) of expressing everything in Euclidean terms on $\R^8$}
will have computational advantages.
\end{remark}

The previous proof related expressions for metrics in Euclidean
coordinates on $\R^8$ to those arising from Riemannian submersions
from $\CP^3$ to $S^4$. We can apply the same technique to the
differential forms defining $\2$ structures. Consider first the 2-form
\[ \tu_0 = dR\we\al_1 - \al_2\we\al_3\] 
defined on $\R^8$. In contrast to $\tu_1,\tu_2,\tu_3,$ its restriction
to the fibres of $\H^2\smz\to\HP^1$ is non-degenerate. Since
\[\ts \frac13R^3\om + R\tu_0 
= \frac13R^3d\hal_1+R\,dR\we\al_1\y = d(\frac13R^3\hal_1)\] 
is exact, it follows from Proposition \ref{exact} that the 2-form
$-R\tu_0$ is an alternative primitive for $\varphi$. Moreover,
\be{ga2}\ba{rcl} d\tu_0
&=& d(R\,dR\we\hal_1 - R^2\,\hal_2\we\hal_3)\y
&=& -R\,dR\we d\hal_1 - 2R\,dR\we\hal_{23} -
     R^2(\tu_2\we\hal_3-\hal_2\we\tu_3)\y
&=& (-R\,dR)\we\tu_1+(-R\al_3)\we\tu_2+(R\al_2)\we\tu_3.
\ea\ee
{\gr Set $\ga_2=d\tu_0$} for consistency with \cite[page 842]{BS}, where (with
different notation) it is paired with the simple 3-form
\[ \ga_1 = (-R\,dR)\we(-R\al_3)\we(R\al_2) = -R^3dR\we\al_{23}.\]
At any given point of $\sC,$ the 3-form $\al^{}_{123}$ generates the
cotangent space to the $S^3$ fibres of $S^7\to S^4,$ whereas $\ga_1$
generates the $\R^3$ fibres.

We have \be{match} -\varphi = d(R\tu_0) = R^{-3}\ga_1 + R\ga_2,\ee which
matches \cite[Case iii, page 844]{BS} with $r=R^4,$ $\kappa=1/2,$ and an
overall change of orientation. The description \eqref{match} enables us to
modify $\varphi$ when the conical metric $\hh_2$ is deformed to the complete
`Bryant-Salamon' metric. The latter has the effect of smoothing the vertex of
the cone, which is replaced by the zero section $S^4$ in $\LaS$. {\gr The next
  statement is well known \cite{BS,GPP}, but the theorem serves to record the
  complete $\2$ structure in a novel way:}

\begin{theorem}\label{BS}
The total space $\widetilde\sC$ of $\LaS$ admits a complete metric
$\BS$ with holonomy equal to $\2$. If the scalar curvature
of $S^4$ equals $1/6,$ it satisfies
\[\gr R^2\BS=(R^4+1)^{1/2}h_2-(R^4+1)^{-1/2}(dR^2+\al_2^2+\al_3^2),\]
and is associated to the 3-form
$-\BSphi = d\left((R^4+1)^{1/4}\tu_0\right)$. 
\end{theorem}

\begin{proof}
The assumption on the scalar curvature means $\kappa=1/2,$ then \cite[page
  844]{BS} tells us that the 3-form associated to the complete $\2$ metric
equals
\[ (R^4+1)^{-3/4}\ga_1 + (R^4+1)^{1/4}\ga_2,\]
which coincides with $-\BSphi$ defined in the theorem. The metric $\BS$ can be
gleaned from the proof of Proposition \ref{hc} with $c=2$. It is represented by
\[ (R^4+1)^{-1/2}R^2(dR^2+\al_2^2+\al_3^2)+(R^4+1)^{1/2}\pi^*\!s_4,\]
relative to the fibration $\LaS\to S^4,$ and one can then express
$\pi^*\!s_4$ in terms of $\hh_2$.
\end{proof}

\setcounter{equation}0
\section{A 2-torus action on $\R^8$}\label{T2}

\emph{Left} multiplication by $\1$ on $\R^8$ gives a Killing vector field
\[ X=X_1= -\x_1\pd_0+\x_0\pd_1-\x_3\pd_2+\x_2\pd_3
          -\x_5\pd_4+\x_4\pd_5-\x_7\pd_6+\x_6\pd_7.\]
Given the sign changes in passing from $Y_1$ to $X,$ we have
\[\ba{rcl}
\frac12(X+Y_1) &=& -\x_1\pd_0+\x_0\pd_1-\x_5\pd_4+\x_4\pd_5,\y
\frac12(X-Y_1) &=& -\x_3\pd_2+\x_2\pd_3-\x_7\pd_6+\x_7\pd_6.
\ea\]
so that these combinations define standard $\1$ actions on the two
summands
\[ \R^4_{0145}=\{(\x_0,\x_1,\x_4,\x_5)\},\quad
   \R^4_{2367}=\{(\x_2,\x_3,\x_6,\x_7)\}.\]
We now consider the associated moment maps.

Although $\R^4_{0145}$ is not a quaternionic subspace of $\H^2$ as
given, we can identify it with $\H$ by setting
\[ q=\x_0\+\x_1i\+\x_4j\+\x_5k.\]
This enables us to apply the analysis from Section \ref{GH}, merely
replacing the indices $2,3$ by $4,5$.  The $\1$ action corresponding
to $\frac12(X+Y_1)$ is triholomorphic, and the moment mapping is
\[ (\x_0,\x_1,\x_4,\x_5)\lmt(\u_1,\u_2,\u_3),\]
where
\be{uuu}\left\{\ \ba{rclcl}
\u_1 &=& x_0^2 + x_1^2 - x_4^2 - x_5^2 &=& |\z_0|^2{\gr-|\z_2|^2}\y
\u_2 &=& 2(\x_0\x_4 + \x_1\x_5) &=& \ {\gr2}\,\Re(\z_0\oz_2)\y
\u_3 &=& 2(\x_0\x_5 - \x_1\x_4) &=& -{\gr2}\,\Im(\z_0\oz_2),\ea\right.\ee
We denote $(\u_1,\u_2,\u_3)$ by $\bfu$. There is a bijection
\[ \R_{0145}^4/\1\cong \La^2_-(\R_{0145}^4)\cong\R^3,\]
which is a diffeomorphism away from the respective origins. We set
\[ u=x_0^2\+x_1^2\+x_4^2\+x_5^2,\]
so that
\[\ts u^2=\suml_{i=1}^3\!u_i^2=|\bfu|^2.\] 
The function $u$ was denoted $\u_0$ in Section \ref{GH}; it will be
convenient to omit the subscript since no confusion should arise in
this printed document with vector $\bfu\in\R^3$.

Similarly, the hyperk\"ahler moment map for $\frac12(X-Y_1)$ on
$\R^4_{2367}$ can be identified with $(\v_1,\v_2,\v_3),$ where
\be{vvv}\left\{\ \ba{rclcl} 
\v_1 &=& x_2^2 + x_3^2 - x_6^2 - x_7^2 &=& |\z_1|^2-|\z_3|^2\y
\v_2 &=& 2(\x_2\x_6 + \x_3\x_7) &=& \ {\gr2}\Re(\z_1\oz_3)\y
\v_3 &=& 2(\x_2\x_7 - \x_3\x_6) &=& \ {\gr2}\Im(\z_1\oz_3).\ea\right.\ee
{\gr Note contrasting signs for the complex expressions of $\u_3$ and $\v_3$.}
We also set $\bfv=(\v_1,\v_2,\v_3)$ and $v=|\bfv|,$ so that
\[\ts u+v = \suml_{i=0}^7 x_i^2 = R.\] 
Our real 6-dimensional quotient space is 
\[ \sM\ =\ \frac{\CP^3}{SO(2)}\times\R^+\ \cong\ 
\frac{\R^8\smz}{\1\times \1}\ \cong\ \R^6\smz,\]
and we shall work with the coordinates 
\[ \buv=(\u_1,\u_2,u_3;\,\v_1,\v_2,\v_3)\] 
on $\sM$. Note that the `radii' $u$ and $v$ are not everywhere smooth in
these coordinates.

\begin{remark}\label{ambi}\rm
The appearance of complex conjugates in the moment maps is a
consequence of the choice \eqref{zzzz}. In fact, $\sM$ is
ambidextrous: as a $T^2$ quotient of $\R^8$ the two $\1$ factors have
equal status. Switching factors amounts to swapping $X$ and $Y_1,$ and
changing the sign of $\bfv$. This is achieved by replacing $\z_1$ by
$-\oz_3$ and $\z_3$ by $\oz_1$. However, it is the additional
structure that we impose that breaks the symmetry. Referring to the
definitions at the start of Section \ref{holo}, we see that a change
from \emph{right} to \emph{left} quaternionic multiplication will
replace $\z_1,\z_3$ by their conjugates and change the sign of $\v_3$
(but not $\v_1,\v_2$).
\end{remark}

The interaction of $X$ with the underlying quaternionic structure of
$\R^8$ gives rise to functions $\mu_j=X\ip\al_j,$ explicitly
\be{mus}\left\{\ba{rcl}
\mu_1 &=& x_0^2+x_1^2-x_2^2-x_3^2+x_4^2+x_5^2-x_6^2-x_7^2\y \mu_2 &=&
2(-\x_0\x_3+\x_1\x_2-\x_4\x_7+\x_5\x_6)\y \mu_3 &=&
2(\x_0\x_2+\x_3\x_1+\x_4\x_6+\x_5\x_7).\ea\right.\ee
These functions will be needed to define an appropriate $Sp(2)$
invariant connection on the $\1$ bundle $\pq$. Observe that
$\mu_1=u-v,$ whereas $\mu_2,\mu_3$ will be used (in
Corollary \ref{hor}) to characterize those points of $\sM$ over which
the circle fibres are \emph{horizontal} in $\CP^3$.

Fix $c<3$. The Riemannian metric $\hh_c$ defined on $\sC$ by
Proposition \ref{hc} can be used to define a 1-form
\be{Thc}\Th_c = \frac2{N_c}\,X\ip\hh_c,\ee
where
\be{Nc} N_c = \hh_c(X,X) =
2(R^2 -\mu_1^2) + {\gr(1-c)}\big(\mu_2^2+\mu_3^2\big).\ee
By design, $\Th_c$ annihilates the orthogonal complement of $X$ and
$X\ip\Th_c=2$. Obviously, $\sL_X\Th_c=0,$ and so we also have
$\sL_X(d\Th_c)=-d(X\ip\Th_c)=0$. This confirms that $d\Th_c$ passes to
the quotient $\sM$. We may regard $\Th_c$ as the connection defined by
the respective metric on the total space of $\pq,$ and $F=d\Th_c$ as
its curvature. Strictly speaking both should be multiplied by
$i\in\fu(1),$ but will work with the real forms. These forms do not
change when the metric is rescaled by a constant.

The factor `2' has been inserted in the definition of $\Th_c$ to
reflect the fact that $\1$ acts on $\R^8,$ while $SO(2)=\1/\Z_2$ acts
effectively on $\CP^3$ and $S^4$. (Up to now, we have blurred this
distinction.) Left multiplication by $e^{it}$ is only effective on
$\CP^3$ for $t\in[0,\pi),$ and the connection 1-form is normalized so
  that the integral of $\Th_c$ over each circle fibre of $\sC$ equals
  $2\pi$. This will be important in a subsequent discussion of Chern
  classes.

Returning to Proposition \ref{hc} and the definitions \eqref{Thc}
and \eqref{Nc}, we infer

\begin{proposition}\label{th}
Assuming $c<3,$ we have
\[\ba{rcl}
N_c &=& 2(5-c)uv +2(1-c)\cuv\y
N_c\Th_c &=& 4R\,X\ip e -4\mu_1\al_1 +2(1-c)(\mu_2\al_2+\mu_3\al_3).
\ea\]
\end{proposition}

Recall that the Riemannian metric
\[ \hh_c = dR^2 + R^2\,\wh_c\]
on the cone $\sC=\R^+\ti\CP^3$ is defined for $c<3,$ and has holonomy
$\2$ when $c=2$.

\begin{definition}\label{ggc}
The \emph{pushdown} of $\hh_c$ is the Riemannian metric $\gg_c$ on
$\sM$ defined by setting
\[\hh_c = \pq^*\gg_c + \qart\hh_c(X,X)\,\Th_c\ot\Th_c.\]
\end{definition}

\noindent To obtain $\gg_c$ one `subtracts' the component of $\hh_c$
tangent to the circle fibres, and
\[\gr \pq\colon(\sC,\hh_c)\lra(\sM,\gg_c)\]
is a Riemannian submersion on an open subset of its domain.\smallbreak

Note that $\gg_c$ is $\1$-invariant and `horizontal' in the sense
that $X\ip(\pq^*\gg_c)=0$. It follows that $\gg_c$ can (for $c$ fixed)
be expressed as
\be{ABC}\ts \suml_{i=1}^3 A_i^{}du_i^2+\suml_{i,j=1}^3 B_{ij}^{}d\u_id\v_j+
\suml_{j=1}^3 C_j^{}dv_j^2,\ee
where the coefficients are rational functions of $\u_1,\ldots,\v_3$.
Our aim is to describe the metrics $\gg_1$ and $\gg_2$ in this way,
but it is instructive first to write down the Euclidean metric $\eu$
on $\R^8$ in terms of our $T^2$ quotient, using

\begin{notation}\label{prime}\rm
Write
\[\al_i=\al_i'-\al_i'',\qquad X\ip\eu=\al_i'+\al_i'',\]
where 
\[\al_i'\in\La^2(\R^4_{0145})^*,\quad\al_i''\in\La^2(\R^4_{2367})^*.\]
Observe that
\[ X\ip\al_1=u-v=\mu_1,\]
whereas $\eu(X,X)=u+v=R$.
\end{notation}

\begin{lemma}\label{eT2} 
The Euclidean metric on $\R^8$ can be written
\[ e =  R^{-1}\big(\al_1{}^{\!2} + uv\,\Th_1{}^{\!2}\big) +
\qart\big(u^{-1}|d\bfu|^2+v^{-1}|d\bfv|^2\big).\]
\end{lemma}

\begin{proof} 
Using the expression for the Gibbons-Hawking metric just above Lemma
\ref{star} and Notation \ref{prime}, we have
\[ e = u^{-1}(\al_1')^2+v^{-1}(\al_1'')^2 + 
\qart\big(u^{-1}|d\bfu|^2+v^{-1}|d\bfv|^2\big).\]
We deduce that
\[ X\ip\hh_1 = 4(u\,\al_1''+v\,\al_1'),\]
which implies that $\Th_1 = u^{-1}\al_1'+v^{-1}\al_1''$ is the sum of
two connection 1-forms of the type defined in Section
\ref{GH} for the 4-dimensional case. The formula for $\eu$ now
follows.
\end{proof}

\vs

\begin{remark}\rm
The lemma provides a non-flat instance of the Gibbons-Hawking ansatz
to supplement our treatment in Section \ref{GH}. If we regard
$\R^8=\H^2$ as a hyperk\"ahler space defined by the \emph{left} action
of $Sp(1)$ then $\bfu-\bfv\colon\R^6\to\R^3$ is a moment mapping for
the triholomorphic action of $\1_1$ generated by $Y_1$. It is well
known that the resulting hyperk\"ahler quotient
\[ H(\bfm) = \frac{\{\bfx\in\R^8:\bfu-\bfv=\bfm\}}{\1_1}\]
is an Eguchi-Hanson space, provided $\bfm\ne\bf0$. Its Ricci-flat metric can be
obtained from Lemma \ref{eT2} by subtracting the vertical term
$R^{-1}\al_1{}^{\!2}$ and setting $\gr d\bfv=d\bfu$. It equals
\[
e - R^{-1}\al_1{}^{\!2}
= R^{-1}uv\,\Th_1{}^{\!2} + \qart(u^{-1}+v^{-1})|d\bfu|^2\yy
= V^{-1}\Th_1{}^{\!2} + \qart V|d\bfu|^2,\]
where
\[ V = u^{-1}+v^{-1} = \frac1{|\bfu|}+\frac1{|\bfu-\bfm|}.\]
By comparison with the formula just before Lemma \ref{star}, we see
that this is a Gibbons-Hawking potential $V$ defined by poles at
$\bf0$ and $\bfm$ in $\R^3$.

In conclusion, we can say that each affine subvariety
\be{affine}
H(\bfm)/\1=\{\buv: \bfu-\bfv=\bfm\ne\bf0\}\subset\sM
\ee
(with $\bfm$ fixed) is the base of an Eguchi-Hanson space. The
corresponding subset of the $(u,v)$-quadrant is the semi-infinite
rectangle
\[ \{(u,v):-m\le u-v\le m\le u+v\}\subset\R^2\]
where $m=|\bfm|$. The function $\bfu=(\u_1,\u_2,\u_3)$ is a moment
mapping for the action of $\1$ generated by $X$. In the description of
$H(\bfm)$ as the cotangent bundle of a 2-sphere $S^2,$ $\1$ rotates
$S^2$ and the poles of $V$ are the fixed points of $S^2$ at the zero
section of $T^*S^2$. For $m=0,$ the space $H(\bf0)$ can be identified
with a cone over $S^3$ which the Hopf map projects to $\Sf,$ the fixed
point set of $\1$ acting $S^4$. This cone is resolved to $T^*S^2$ when
$\sC$ is deformed into $\widetilde\sC$ (cf.\ Theorem \ref{BS}).
\end{remark}

Returning to Definition \ref{ggc}, the K\"ahler case $c=1$ is easy to
describe, since the splitting $\R^8=\R^4_{0145}\op\R^4_{2367}$ is
preserved:

\begin{theorem}\label{g1}
\[ \gg_1 = \half dR^2 + 
\half R\big(u^{-1}|d\bfu|^2+v^{-1}|d\bfv|^2\big),\]
where $|\bfA|^2$ denotes $\bfA\cd\bfA$.
\end{theorem}

\begin{proof}
It now suffices from Proposition \ref{hc} to verify that
\[ \hh_1(X,X)\Th_1{}^{\!2} =
2R\big(u^{-1}(\al_1')^2+v^{-1}(\al_1'')^2\big) - 2\al_1^2,\] 
and this follows from equations in the proof of Lemma~\ref{eT2}.
\end{proof}

The case $c=2$ is more complicated, and some preliminary definitions
will render the result more transparent. Working on $\sM$ away from
the locus $uv=0,$ we define vectors
\be{A+-} \bfA_\pm = v\bfu\mp u\bfv\ee
and scalars
\be{a+-} \aa_\pm = uv\mp \cuv.\ee
If we denote by $2\th$ the angle between $\bfu$ and $\bfv$ (for
$0\le\th\le\pi/2$) so that $\cuv=uv\cos2\th,$ then
\[\ba{rcl}
\bfA_+\cd\bfA_+=2uv\,\aa_+ &\Rightarrow& |\bfA_+|=2uv\sin\th\y
\bfA_-\cd\bfA_-=2uv\,\aa_- &\Rightarrow& |\bfA_-|=2uv\cos\th,
\ea\]
Moreover, $\bfA_+\cd\bfA_-=0$.

The next definition will be exploited repeatedly in the sequel:

\begin{definition}\label{align}
Set
\[ \sF_\pm=\{\buv\in\sM: \bfA_\pm=\bf0\}.\]
\end{definition}

\noindent Note that $\sF_+\cap\sF_-=\{\buv\in\sM:uv=0\},$ which (modulo the
origin of $\R^6$) is $\R^3\cup\R^3$. The formulae above make it clear
that, away from their intersection, $\sF_+$ (resp.\ $\sF_-$) consists
of points $\buv$ for which $\bfu,\bfv$ are parallel and aligned
(resp.\ anti-aligned). This explains our choice of opposing signs in
the definitions \eqref{A+-} and \eqref{a+-}.\smallbreak

For the purpose of describing $\gg_2,$ we also define a vector-valued
1-form
\be{B} \bfB = u\,d\bfv -v\,d\bfu\ee
and scalar 1-forms
\be{Ga+-} \Ga_\pm = \frac1{uv}\bfA_\pm\cd\bfB,\ee
so that $\Ga_\pm$ vanishes on $\sF_\pm$.

\begin{lemma}\label{Ga}
\[\ba{rcccl}
\Ga_+ &=& \gr -u\,dv - v\,du + \bfu\cd d\bfv + \bfv\cd d\bfu
&=&\gr -d\aa_+\y
\Ga_- &=& u\,dv - v\,du + \bfu\cd d\bfv - \bfv\cd d\bfu
&=& -2(\mu_2\al_3-\mu_3\al_2).\ea\]
\end{lemma}

\begin{proof}
The 1-form $\mu_2\al_3-\mu_3\al_2$ is initially defined on $\sC,$ but it is
$\1$-invariant and has zero contraction with $X,$ so passes to $\sM$. The final
equality now follows from a direct computation, whilst the others are more
elementary.
\end{proof}

These definitions enable us to state and prove

\begin{theorem}\label{g2}
\[ \gg_2 = \half dR^2 + \half\big|d\bfu+d\bfv\big|^2 + 
\frac2N|\bfB|^2 + \frac1{2N}\Ga_+^2 - \frac1{4N}\Ga_-^2,\]
where $N=N_2=6uv-2\cuv$.
\end{theorem}

\noindent Note that $\gg_2$ is, as its definition requires, well
defined where $uv\ne0$. The latter actually implies that $N_c>0$ for
$c<3$; this follows from Proposition \ref{th} since
$5-c>|1-c|$. \smallbreak

Theorem \ref{g2} was originally derived by solving equations for
$A_i,B_{ij},C_j$ in \eqref{ABC} with the help of Maple. We shall
present a rigorous proof based on a formula for $\gg_c$ which is less
obviously non-singular:

\begin{lemma}\label{gc}
The metric $\gg_c$ induced on $\sM$ by the conical metric $\hh_c$ on
$\sC$ equals
\[ \gg_1 + (1-c)\left[\frac1{8\aa_-}\Ga_-\!^2 +
\frac1{uvN_c\aa_-}\left\{\bfB,\bfu,\bfv\right\}^2\right],\]
where curly brackets indicate triple product.
\end{lemma}

\begin{proof}
It will be convenient to define the 1-form
$\eta=\mu_2\al_2+\mu_3\al_3$ on $\sC$ for the scope of the proof.
{\gr From Proposition \ref{th},
\[ N_c\Th_c = N_1\Th_1+2(1-c)\eta = 8uv\,\Th_1 + 2(1-c)\eta.\]}
A computation also gives $\mu_2^2+\mu_3^2 = 2\aa_-,$ so Lemma \ref{Ga}
implies
\be{eta2}\ba{rcl} \eta^2 + \qart\Ga_-\!^2
&=& \eta^2+(\mu_2\al_3-\mu_3\al_2)^2\y
&=& (\mu_2^2+\mu_3^2)(\al_2^2+\al_3^2)\y
&=& 2\aa_-(\al_2^2+\al_3^2).\ea\ee
Using Proposition \ref{hc},
\[\ba{rcl} \gg_c
&=& \hh_c - \qart h_c(X,X)\Th_c\!^2\yyy
&=& \hh_1 +(1\-c)(\al_2^2+\al_3^2) - \qart N_c\Th_c\!^2\y
&=& \ds \gg_1 + 2uv\Th_1^2 +(1\-c)(\al_2^2+\al_3^2)
    - \frac1{N_c}\big(4uv\,\Th_1 + (1\-c)\eta\big)^2.\ea\]
Using \eqref{eta2} to eliminate $\al_2^2+\al_3^2,$ and the equation
$N_c = 8uv +2(1\-c)\aa_-,$ we obtain
\[ \gg_c = \gg_1 + \frac1{8\aa_-}(1\-c)\Ga_-\!^2 
+ (1\-c)\frac{4uv}{\aa_-N_c}(\aa_-\Th_1 - \eta)^2.\]
The 1-form $\aa_-\Th_1 - \eta$ has zero contraction with $X$ and is $\1$
invariant, and must therefore be expressible in terms of $\bfu$ and
$\bfv$. A direct calculation yields
\be{curly} 2(\aa_-\Th_1 - \eta) = \frac1{uv}\T \bfB,\bfu,\bfv,,\ee
using Notation \ref{triple} from the next section.
\end{proof}

\begin{proof}[Proof of Theorem \ref{g2}]
In the light of Theorem \ref{g1}, we need to show that
\[ \gg_1-\gg_2 = \frac1{2uv}|\bfB|^2 - \frac2N|\bfB|^2 -
\frac1{2N}\Ga_+\!^2+\frac1{4N}\Ga_-\!^2,\]
where $N=N_2,$ or equivalently
\[ 4Nu^2v^2(\gg_1-\gg_2) = 4uv\,\aa_+|\bfB|^2 -
2(\bfA_+\cd\bfB)^2 + (\bfA_-\cd\bfB)^2.\]
By Lemma \ref{gc}, it suffices to show that the right-hand side of
this last equation equals
\[ \frac N{2a_-}(\bfA_-\cd\bfB)^2 +
\frac4{uva_-}\left\{\bfB,\bfu,\bfv\right\}^2.\]
{\gr Noting that $\half N-a_-=2a_+,$ this is equivalent to the assertion that}
\be{BBBB}\gr
2uv\big[\aa_+\aa_-|\bfB|^2 - \left\{\bfB,\bfu,\bfv\right\}^2\big]
= \aa_-(\bfA_+\cd\bfB)^2 + \aa_+(\bfA_-\cd\bfB)^2.\ee
This identity can be verified by inspection of its components with respect to
$B_1^2$ and $B_1B_2,$ where $\bfB=(B_1,B_2,B_3)$. (The $B_i$ are linearly
independent 1-forms provided $uv\ne0$.) Dividing by $\gr 2uv,$ this gives the
two equations
\[\ba{rcl}
|\bfu\times\bfv|^2 -(\u_2\v_3-u_3\v_2)^2 &=&
u^2v_1^2+v^2u_1^2-2\u_1\v_1\cuv,\y 
-(\u_2\v_3-u_3\v_2)(\u_3\v_1-u_1\v_3) &=&
u^2\v_1\v_2+v^2\u_1\u_2-(\u_1\v_2+\v_1\u_2)\cuv,\ea\]
whose validity is readily checked.
\end{proof}

Lemma \ref{gc} tells us that $\gg_1-\gg_2$ belongs pointwise to the
6-dimensional space spanned by the quadratic forms
\[ B_iB_j =
u^2d\v_id\v_j-uv(d\u_id\v_j+d\v_id\u_j)+v^2d\u_id\u_j.\]
We shall exploit Theorem \ref{g2} in Sections \ref{SU3} and \ref{met}.

\setcounter{equation}0
\section{SO(3) invariance}\label{curv}

The previous results have revealed the evident $SO(3)$ symmetry
inherent in the bivector formalism with $\buv\in\R^6$. In this
section, we investigate the effect of this and other symmetries, and
study the curvature of the $\1$ bundle defined by \eqref{Q} away from
the singular locus $\R^3\cup\R^3$ of $\R^6$.

The $\2$ structure on $\sC=\R^+\ti\CP^3$ is invariant by $SO(5),$
which is double covered by the action of $Sp(2)$ on $\H^2$. The
diagonal $\1$ in $Sp(2)$ commutes with $SU(2),$ which acts on $\H^2$
by
\[ \left(\!\!\ba{c} q_0\\q_1\ea\!\!\right)\lmt
A\left(\!\!\ba{c}q_0\\q_1\ea\!\!\right),\qquad
A=\left(\!\ba{cc}a&b\\-\ol b&\ol a\ea\!\right),\ |a|^2\+|b|^2=1,\]
and on $\C^4$ by
\[\left(\!\ba{cc} z_0&-z_3\\z_2&z_1\ea\!\right)\lmt
A\left(\!\ba{cc} z_0&-z_3\\z_2&z_1\ea\!\right).\] 
The determinant $\z_0\z_1+\z_2\z_3$ will play a key role in the
sequel (see Proposition \ref{pD}).  

The map \eqref{Q} is induced from mapping
$\bfq=(\q_0,\>\q_1)^{\!\top}$ to the pair of Hermitian matrices
\be{UV}
\wfu=\left(\!\ba{cc}u_1&u_2\-iu_3\y u_2\+iu_3&\-u_1\ea\!\right),
\quad
\wfv=\left(\!\ba{cc}v_1&v_2\-iv_3\y v_2\+iv_3&\-v_1\ea\!\right).
\ee
This representation helps to explain our convention in the definition
of the Gibbons-Hawking coordinates $u_i,v_j$ in Section \ref{T2}. In
any case,
\[ \wfu(A\bfq)=A\,\wfu A^{-1},\quad\wfv(A\bfq)=A\,\wfv A^{-1},\]
either of which induces the usual double covering $SU(2)\to\3$. These
facts tell us that the residual subgroup $\3=SU(2)/\Z_2$ acts
diagonally on $\sM\subset\R^3\times\R^3$. As a subgroup of $SO(5),$ it
acts trivially on a 2-dimensional subspace $\R^2$ in $\R^5$ that (we
shall see) corresponds to the subset $\{\bfu=-\bfv\}$ of $\sM$.

The right action of $j$ on $\H^2$ induces an anti-linear involution of
$(\CP^3,J_1)$ without fixed points; it acts as the antipodal map on
each $S^2$ fibre. It is the so-called \emph{real} structure, and $S^4$
can be defined as the set of \emph{real} (meaning $j$-invariant)
projective lines in $\CP^3$. Since $Y_2$ is the Killing vector field
generated by the action of $e^{jt},$ we can compute the action of $j$
at a given point by applying the associated rotation by $\pi/2$. 
If we define 
\[ j^*\bfz=\bfz\circ j^{-1}=-\bfz\circ j,\]
then \eqref{qq} tells us that
\be{j*z}\ba{c}
j^*(\z_0,\z_1,\z_2,\z_3)=(\oz_1,-\oz_0,\oz_3,-\oz_2)\yy
j^*(\x_0,\x_1,\x_2,\x_3,\x_4,\x_5,\x_6,\x_7)=
(\x_2,\x_3,-\x_0,-\x_1,\x_6,\x_7,-\x_4,-\x_5).\ea\ee 
The left $\1$ action commutes with $j,$ and we have

\begin{lemma}\label{j}
The involution $j$ passes to an isometry $\hj\colon\sM\to\sM$ that interchanges
the coordinates $\bfu\leftrightarrow\bfv$. Under this operation, the curvature
$2$-form $\gr F_c=d\Th_c$ is symmetric. As regards the tensors in Lemma
\ref{FF}, the symplectic form $\si$ is anti-symmetric, and the $(3,0)$-form
$\Psi$ maps to its complex conjugate $\ol\Psi$.
\end{lemma}

\begin{proof} 
The fact that $\hj^*\bfu=\bfv$ follows from \eqref{j*z}. Since $j$
leaves invariant $X$ and
\[ j^*\al_1=-\al_1,\quad j^*\al_2=\al_2,\quad j^*\al_3=-\al_3,\]
it has the same $(-_,+,-)$ effect on the functions $\mu_i$. Hence
$j^*\Th_c=\Th_c,$ $j^*\om=-\om,$ and $j^*\Upsilon=\ol\Upsilon$.
(Note that $j$ must change the sign of $\om$ on $\CP^3$ because it
is an isometry for both the K\"ahler and nearly-K\"ahler metric on
$\CP^3,$ and yet changes the sign of both $J_1,J_2$.) We conclude that
\[ j^*\varphi=-\varphi,\qquad j^*(\vasphi)=\vasphi,\]
and the rest follows.
\end{proof}

Consider the connections described by Proposition \ref{th}. We shall
first show the curvature 2-form
\be{Fc} F_c = d\Th_c = d\Big(\frac2{N_c}X\ip\hh_c\Big)\ee
is completely determined (for any $c$) by its restriction to
$S^2\times S^2$ in $\R^6$. Consider the action of $\R^+\times\R^+$ on
$\R^4_{0145}\times\R^4_{2367}$ given by
\[ (\x_0,\x_1,\x_2,\x_3,\x_4,\x_5,\x_6,\x_7)\lmt
(\la\x_0,\la\x_1,\mu\x_2,\mu\x_3,\la\x_4,\la\x_5,\mu\x_6,\mu\x_7),\]
with $\la,\mu>0$. We denote this representation by $\R^{++}$. The
infinitesimal action is generated by the vector fields
\[\ba{rll}
\pd_u &=\ u\frac\pd{\pd u}
&=\ \x_0\pd_0+\x_1\pd_1+\x_4\pd_4+\x_5\pd_5,\yy \pd_v
&=\ v\frac\pd{\pd v} &=\ \x_2\pd_2+\x_3\pd_3+\x_6\pd_6+\x_7\pd_7,\ea\]
Note that $\pd_u+\pd_v$ is dual to $dR$ relative to the Euclidean
metric on $\R^8$. 

The invariance of the curvature form by $\R^{++}$ derives from that
of its primitive:

\begin{lemma}\label{++}
For any $c,$ the 1-form $\Th_c$ is invariant by $\R^{++}$. 
\end{lemma}

\begin{proof}
It suffices to consider the action of $\la$ above, and to do this we
use Notation \ref{prime}. The action on $\R^4_{0145}$ has the
following effect on tensors:
\[\ba{c}
\ba{rcl}
 R &\mapsto& \la^2u+v\\
 X\ip e &\mapsto& \gr \la^2\al_1'+\al_1''\\
 N_c &\mapsto& \la^2N_c,\quad c=1,2,\ea\\
\ba{ccc}
\mu_1 \mapsto \la^2u-v, &\qquad& \al_1 \ \mapsto\  \la^2\al_1'-\al_1''\\
\mu_2,\mu_3\mapsto\la\mu_2,\la\mu_3, &&
\al_2,\al_3 \mapsto\la\al_2,\la\al_3.\ea
\ea\]
It now follows that both $R(X\ip e)-\mu_1\al_1$ and
$\mu_2\al_2+\mu_3\al_3$ scale homogeneously by $\la^2,$ which is
cancelled out by dividing by $N_c$.
\end{proof}

The \emph{K\"ahler quotient} of $\CP^3$ by $\1$ is constructed by
first identifying the moment mapping $f$ defined by
\[ df = X\ip d\hal_1 = -d(X\ip\hal_1),\]
$d\hal_1$ being the K\"ahler form. Hence,
\[ -f=\frac{\mu_1}R = \frac{u-v}{u+v} =
\frac{|z_0|^2-|z_1|^2+|z_2|^2-|z_3|^2}{|z_0|^2+|z_1|^2+|z_2|^2+|z_3|^2}\]
(see \eqref{mus}). If we regard $\CP^3$ as the hypersurface of $\sC$
defined by $R=1$ then we can identify the quotient at level
$\ell\in[-1,1]$ with
\[ \frac{f^{-1}(\ell)}\1 =\{\buv\in\sM: -u\+v=\ell,\ u\+v=1\},\] 
On the open interval this is $S^2\times S^2,$ with the respective
sphere collapsing as $c\to\pm1$. A product K\"ahler 2-form on
$S^2\times S^2$ then pulls back to the restriction of $d\hal_1$ to
$f^{-1}(\ell),$ whereas
\[ J_1d\mu_1 = \half X\ip\hh_1 = 2(v\al_1'+u\al_1''),\]
in the notation of Proposition \ref{++}.

By general principles, we can identify a generic K\"ahler quotient by
a compact Lie group $G$ with a stable holomorphic quotient by $G^c$
\cite{Kir}. In our case, the $\1$ action on $\CP^3$ obviously extends
to
\[ \{[\z_0,\z_1,\z_2,\z_3]\mapsto
[\zeta\z_0,\zeta^{-1}\z_1,\zeta\z_2,\zeta^{-1}\z_3],
\quad\zeta\in\C^*.\] 
Mapping $[\z_0,\z_1,\z_2,\z_3]$ to
$[\z_0\z_1,\z_2\z_3,\z_1\z_2,\z_0\z_3]$ then realizes the $\C^*$
quotient as a quadric biholomorphic to $\CP^1\times\CP^1$.

Next we shall express $F_c$ in terms of $\3$ invariant quantities
manufactured from the coordinates $\buv$ using scalar and triple
products. We had originally carried this out only for $c=2,$ but the
general case enables us to express the relationship with $F_1,$ by
analogy to Lemma \ref{gc}. The formulae will also include the radii
$u,v,$ which (as we have remarked) are not smooth over $\R^3\cup\R^3$.

\begin{notation}\label{triple}\rm
In order to state results in this section (and adjacent ones), we
exploit various triple products combining functions and 1-forms. Our
convention is that each triple product has $3!=6$ terms in which a
wedge product counts as one. This is exemplified by the following
dictionary:
\[\ba{rcl}
\T\bfu,\bfv,d\bfu,    &=& (\u_2v_3-u_3v_2)d\u_1+\cdots\\
\T\bfu,d\bfu,d\bfu,   &=& 2\,\u_1d\u_2\we d\u_3+\cdots\\
\T d\bfu,d\bfu,d\bfu, &=& 6\,d\u_1\we d\u_2\we d\u_3\\
\T\bfu,d\bfu,d\bfv,   &=& \u_1(d\u_2\we d\v_3-\gr d\u_3\we d\v_2)+\cdots\\
\T\bfu,d\bfv,d\bfv,   &=& 2\,\u_1d\v_2\we d\v_3+\cdots\\
\T d\bfu,d\bfv,d\bfv, &=& 2\,d\u_1\we d\v_2\we d\v_3+\cdots
\ea\]
This is a consistent scheme, in the sense that (for example) the
substitution $\bfv=\bfu$ in $\T\bfu,d\bfu,d\bfv,$ yields
$\T\bfu,d\bfu,d\bfu,$.  Moreover,
\[\ba{rcl}
\T d\bfu,d\bfu,d\bfu, &=& d\T\bfu,d\bfu,d\bfu,\\
\T d\bfu,d\bfv,d\bfv, &=& d\T\bfu,d\bfv,d\bfv,\\
\T \bfu,d\bfu,d\bfv,  &=& d \T\bfu,\bfv,d\bfu,+2\T\bfv,d\bfu,d\bfu,.
\ea\]
\end{notation}

On a separate matter,

\begin{notation}\rm
In the light of Lemma \ref{j}, and to reduce the length of our
displays, we shall use
\[\ba{ll} 
X \ceq Y &\qbox{as shorthand for}X=Y+\widetilde Y,\y 
X\ceqq Y &\qbox{as shorthand for}X=Y-\widetilde Y,\ea\] 
where $\widetilde Y$ denotes $Y$ with $\bfu$ and $\bfv$ (and $u$ and
$v$) interchanged. The second symbol will only be used in
Theorem~\ref{Psi}.
\end{notation}

The conventions above enable us to state

\begin{theorem}\label{F} 
  Assuming $c<3,$ the curvature 2-form $F_c$ defined by \eqref{Fc} is
  given by
\[ N_c\!^2\,F_c\ \ceq\
8(3\-c)\frac{v^2}u \T\bfu,d\bfu,d\bfu, + 4(1\-c)\big[
du\we\T\bfu,\bfv,d\bfv, + u\T\bfv,d\bfu,d\bfv,\big].\]
\end{theorem}

\begin{proof}
From the proof of Theorem \ref{g1},
\[\Th_1 = u^{-1}\al_1'+v^{-1}\al_1''.\]
We conclude that
\[ F_1 = \qart u^{-3}\T\bfu,d\bfu,d\bfu,+\qart v^{-3}\T\bfv,d\bfv,d\bfv,
\ \ceq\ \qart u^{-3}\T\bfu,d\bfu,d\bfu,.\]
This is the sum of the two 4-dimensional curvatures encountered in
Lemma \ref{star}, and (since $N_1=8uv$) agrees with the theorem's
statement.

Lemma \ref{++} tells us that $F_c$ is invariant by the action
$\buv\mapsto(\la^2\bfu,\mu^2\bfv)$ on $\R^6$. A computation shows that
the restriction $\wF_c$ of $F_c$ to $S^2\times S^2$ (so $u=1=v$)
satisfies
\[ N_c\!^2\wF_2\ \ceq\ 
8(3\-c)\T\bfu,d\bfu,d\bfu, + 4(1\-c)\T\bfu,d\bfu,d\bfv,.\] 
To find the expression for $F_c$ on (an open set of) $\R^6$ stated in
the theorem, we replace $\buv$ by $(u^{-1}\bfu,v^{-1}\bfv)$ in $\wF_c$.
\end{proof}

\begin{definition}\label{sMn}
Given a unit vector $\bfn$ in $\R^3,$ set
\[ \sM(\bfn) =
\{\buv\in\sM: \bfu\cd\bfn=0=\bfv\cd\bfn,\ uv\ne0\}.\]
\end{definition}

\noindent Thus, $\sM(\bfn)$ is an open subset of the 4-dimensional
subspace of $\R^6$ that satisfies the vector equation
$(\bfu\times\bfv)\times\bfn=\bf0$. Note that $\sM(\bfn_1)=\sM(\bfn_2)$
if and only if $\bfn_1=\pm\bfn_2$ and that $\sM(\bfn_1)\cap\sM(\bfn_2)$
lies in a plane parametrized by $(\pm u,\pm v)$ otherwise. The
curvature form $F_c$ and the induced $SU(3)$ structure are well
defined at all points of $\sM(\bfn),$ and this subset will play an
important role in both Sections \ref{SU3} and \ref{met}.\smallbreak

\begin{corollary}\label{van}
$F_c$ vanishes on each linear subvariety $\sM(\bfn)$ for any $c$. It
also vanishes on the subset $\{\bfu=-\bfv\}$ of $\sM$ (which by
Definition \ref{align} lies in $\sF_-$).
\end{corollary}

\begin{proof}
These statements are consequences of Theorem \ref{F}. In the first case, $\bfu$
and $\bfv$ are constrained to lie in a common 2-dimensional subspace of $\R^3,$
{\gr as are their derivatives.} It follows that all the triple products
vanish. In the second case, all the triple products change sign (or are zero)
when $\bfu$ and $\bfv$ are interchanged.
\end{proof}

The first Chern class of the bundle $\pq$ over $S^2\times S^2$ is
given by
\[ \chern_1(\pq) = \frac1{2\pi}[d\Th_c] = \frac1{2\pi}[F_c].\]
Now take $c=1$. From the discussion after Lemma \ref{star}, we know
that $\qart\T\bfu,d\bfu,d\bfu,$ integrates to $-2\pi$ over $S^2$. The
bidegree of $\chern_1(\pq)$ over $S^2\times S^2$ is therefore
$(-1,-1)$. The same conclusion must be valid for other values of $c,$
and we can do a consistency check by restricting the formula for $F_2$
to the diagonal sphere $\Delta = \{(\bfu,\bfu):\bfu\in S^2\}$. Bearing
in mind Notation \ref{triple}, we obtain $F_2|_\Delta =
\half\T\bfu,d\bfu,d\bfu,,$ so
\[ \frac1{2\pi}\!\int_\Delta F_2 = -2.\]
This is what one expects since the associated complex line bundle over
$\Delta\cong\CP^1$ is isomorphic to $\sO(-1)\ot\sO(-1)=\sO(-2)$.

Since $F_1$ and $F_c$ are equal in cohomology, their difference must
be exact, and this is made explicit by the next result, whose proof we
omit:

\begin{proposition}
On $S^2\times S^2,$
\[ \wF_c\ \ceq\ \qart\T\bfu,d\bfu,d\bfu, + 
d\Big(\frac{1-c}{(3\-c)N_c}\T\bfu,\bfv,d\bfu,\Big).\]
\end{proposition}

\noindent The restriction of the circle bundle $\pq$ to $S^2\times
S^2$ is homeomorphic to the cone over $S^2\times S^3,$ the basic
Sasaki-Einstein manifold $T^{1,1}$ (see \cite[Appendix A]{GMSW}).
\smallbreak

\begin{remark}\rm
If we choose to identify $\sM$ with $\C^3\smz$ by $\buv\mapsto
\bfu+i\bfv,$ we can define an $\3$-equivariant mapping
$\sM\to\CP^2$. A slice to the $\1$ orbits is given by
\[\{\buv\in\sM: \bfu\cdot\bfv=0\},\]
and the equation $u=v$ then determines a conic curve $C\cong\CP^1$. This set-up
was used by Li to construct a new $\3$-invariant K\"ahler-Einstein metric on
$\CP^2\sm C$ with cone angle $2\pi/3$ along $C$ \cite{Li} (see also
\cite{DS,LS,MSp}). It gives rise to a Sasaki-Einstein metric of
cohomogeneity-one on the link of the singularity $z_1^2+z_2^2+z_3^2+z_4^3=0,$
but this example is not compatible with our geometry; in particular the action
of $\1$ on $\C^3$ is not an isometry for any $c$.
\end{remark}

\setcounter{equation}0
\section{The reduced twistor fibration}\label{SO3}

We shall continue to analyse the action of $\3$ in this section, but
in relation to the twistor fibration $\pi\colon\CP^3\to S^4$.
Consider the lower part of Figure 1 in the Introduction, in which
$ D^3\cong S^4/SO(2)$ is the closed unit ball in the subspace $\R^3$ of
$\R^5$ fixed by the action of $SO(2)=\1/\Z_2$. We may identify its
boundary 2-sphere $\pD$ with the fixed point set $\Sf$ of $SO(2)$ in
$S^4$. The map $\varpi$ can be expressed as a composition
\[\varpi\colon\sM\lra\frac{\sM}{\R^+}\cong\frac{\CP^3}{SO(2)}\lra D^3,\]
in which the second map is a reduction of $\pi$. It is
$\3$-equivariant, and symmetric in $\bfu,\bfv$ since $\pi$ commutes
with $j$. In fact, it could not be simpler:

\begin{proposition}\label{ball} 
  \[\varpi\buv=\frac{\bfu+\bfv}{u+v}.\]
\end{proposition}

\begin{proof}\gr
Recall the antilinear involution $j$ in \eqref{j*z} that defines an
identification $\C^4=\H^2$. Consider the $Sp(2)$-equivariant mapping
$\C^4\to\La^2\C^4$ defined by
\[\ba{rcl} \bfz
&\mapsto& \bfz\we(j^*\bfz)\y
&=& -(\z_0,\z_1,\z_2,\z_3)\we(\oz_1,-\oz_0,\oz_3,-\oz_2)\y 
&=& \big(|\z_0|^2\+|\z_1|^2,\ -\z_0\oz_3\+\z_2\oz_1,\ \z_0\oz_2\+\z_3\oz_1,\
-\z_1\oz_3\-\z_2\oz_0,\ \z_1\oz_2\-\z_3\oz_0,\ |\z_2|^2\+|\z_3|^2\big). 
\ea\]
In the last line, we have used an obvious basis for the exterior product
$\La^2\C^4,$ and we denote the corresponding coordinates by
$\zeta_1,\ldots,\zeta_6,$ so that $\zeta_1+\zeta_6=\|\bfz\|^2,$ $\zeta_2=
\ol{\zeta_5}$ and $\zeta_3=-\ol{\zeta_4}$. These coordinates are unaffected by
changing the phase of $\bfz,$ and
\be{Pi} \pi([\bfz]) = \frac1{\|\bfz\|^2}
\big(2\Re\zeta_5,\ -2\Im\zeta_5,\ \zeta_1-\zeta_6,\ -2\Re\zeta_4,\ -2\Im\zeta_4
\big)\ee
is a unit vector in $\R^5$. The associated representation of $Sp(2)$ on $\R^5$
defines the double covering $Sp(2)\to SO(5),$ and $\pi$ is a well-defined
mapping $\CP^3\to S^4$. It also follows that $\pi$ coincides with the twistor
fibration determined by \eqref{pi}, given that the latter involves the
expression $-\zeta_4-j\zeta_5$.

Note that $\1$ rotates the first two coordinates on the right-hand side of
\eqref{Pi}, and acts trivially on the last three. It follows that $\varpi\buv$
can be identified with the vector in $\R^3$ with coordinates
\[ \frac1{\|\bfz\|^2}\big(\zeta_1-\zeta_6,\ -2\Re\zeta_4,\ -2\Im\zeta_4\big) =
\frac1R(\u_1+\v_1,\ \u_2+\v_2,\ \u_3+\v_3) = \frac{\bfu+\bfv}{u+v}.\]
Our choice of signs in \eqref{Pi} was dictated by the convention \eqref{pi},
and ensures that $\varpi$ is $\3$-equivariant.
\end{proof}

Observe that
\[ \varpi\buv = \frac u{u+v}\bfs+\frac v{u+v}\bft\] lies on the chord
with unit endpoints $\bfs=\bfu/u$ and $\bft=\bfv/v$ in $\pD$. The
generic fibre of $\varpi$ has dimension 3, but reduces to dimension 2
{\gr over} $\pD$. For future reference, we note that the radius
\be{sv}   s=|\varpi\buv|=\frac1R|\bfu+\bfv|\ee
in $\pD$ defines a function $S^4\to[0,1]$. It vanishes on the circle
$\Sm=S^4\cap\R^2,$ which is the fixed point set of the $\3$ action
(and maximal $SO(2)$ orbit), and reaches the extreme value $1$ on the
fixed point set $\Sf$ of the $SO(2)$ action.

Definition \ref{align} underlies many properties of $\varpi$ and $Q,$
and of the tensors defined on the spaces that feature in Figure 1. It
follows from Proposition \ref{ball} that
\be{inv+} \sF_+=\varpi^{-1}(\pD),\ee
and $\pq^{-1}(\sF_+)$ is a cone over the inverse image
\[ \pi^{-1}(\Sf)\cong S^2\times S^2\subset\CP^3\]
that we shall next realize as a non-holomorphic quadric. By contrast,
$\sF_-$ projects \emph{onto} $S^4$ and $ D^3,$ and arises from a
\emph{holomorphic} quadric intimately connected to the $SO(2)$ action
on the 4-sphere.

\begin{proposition}\label{pD} 
 Set $f_+=\z_0\oz_3-\oz_1\z_2$ and $f_-=\z_0\z_1+\z_2\z_3$. Then
\[\pq^{-1}(\sF_\pm) = \big\{[\bfz]\in\sC: f_\pm=0\big\}.\]
\end{proposition}

\newcommand{\X}[2]{\x_{#1#2}}
\newcommand{\XX}[4]{\x_{#1#2#3#4}}

\begin{proof}
Square brackets in the line above represent the $\1_1$ quotient
$\C^4\smz\to\sC$ generated by $Y_1$. We shall in fact prove that
\be{2ff} 2|f_\pm|^2 = uv\mp\cuv,\ee
a quantity that was denoted by $\aa_\pm$ in \eqref{a+-}. This clearly
suffices, and we only need to verify the equation for $f_+$ since the
other then follows from \eqref{vvv}.

Using the Euclidean coordinates on $\R^8$ defined by \eqref{qq}, we
abbreviate $\x_i\x_j$ by $\X ij$ and $\x_i\x_j\x_k\x_l$ by $\XX
ijkl$. As a first step, $uv-\u_1\v_1$ equals
\[(\X00+\X11+\X44+\X55)(\X22+\X33+\X66+\X77) - 
(\X00+\X11-\X44-\X55)(\X22+\X33-\X66-\X77),\]
which is twice
\[\ba{c}
\XX0066+\XX0077+\XX1166+\XX1177+\XX4422+\XX4433+\XX5522+\XX5533\y
\hskip50pt = (\X06-\X17)^2+(\X16+\X07)^2+(\X24-\X35)^2+(\X25+\X34)^2.
\ea\]
On the other hand, $\u_2\v_2+\u_3\v_3$ equals 4 times
\[(\X04+\X15)(\X26+\X37)+(\X05-\X14)(\X27-\X36)
 =(\X06-\X17)(\X24-\X35)+(\X16+\X07)(\X25+\X34).\] 
Therefore $uv-\cuv$ equals twice
 \[ (\X06-\X17-\X24+\X35)^2+(\X16+\X07-\X25-\X34)^2,\] 
 and the two terms in parentheses are the real and imaginary
 components of $f_+$.
\end{proof}

The almost complex structures $J_1,J_2$ (only the first integrable)
and the Einstein metrics $\wh_1,\wh_2$ on $\CP^3$ coincide on the
horizontal distribution
\[ D = \left<\al_2,\al_3\right>^\mathrm{o},\]
which at each point is the orthogonal complement to the vertical
tangent space to the fibration $\pi\colon\CP^3\to S^4,$ with respect
to any of the metrics $\wh_c$. The next result follows from the
equation
  \[ X\ip(\al_2-i\al_3) = \mu_2-i\mu_3 = 2i(\z_0\z_1+\z_2\z_3)\]
(see \eqref{mus}) and Proposition \ref{pD}:

\begin{corollary}\label{hor}
$\sF_-$ is the locus of points in $\sM$ over which the circle fibres
  are horizontal relative to $\pi$.
\end{corollary}

A quadratic form in the $\z_i$'s can be regarded as a holomorphic
section of 
\[ H^0(\CP^3,\sO(2))\cong\sp(2,\C),\] 
which can in turn be identified with the complexification of the Lie
algebra
\[ \sp(2)=\fu(2)\op\fm = \su(2)\op\fu(1)\op\fm\]
of Killing vector fields on $S^4$. Here $\fu(1)$ is spanned by the
vector field $X_*=\pi_*X$ on $S^4$ that generates our circle action,
and $f_-$ is invariant by $U(2)$. The associated quadric $Z_-$ in
$\CP^3$ is the divisor defined by $X_*,$ whose equation is determined
by the self-dual component of the 2-form $\nabla X_*$
\cite[ch.~13]{Besse}.

Now that we have identified the $\fu(1)$ summand, it is easy to see
that Lie subalgebra $\su(2)$ is generated by the quadrics
\[\ba{rcl} 
Z_1 &=& \{\z_0\z_1-\z_2\z_3=0\}\\
Z_2 &=& \{\z_0\z_3+\z_1\z_2=0\}\\
Z_3 &=& \{\z_0\z_3-\z_1\z_2=0\}.\ea\]
All four equations representing $Z_-,Z_1,Z_2,Z_3$ are $j$-invariant,
which is equivalent to asserting that they define \emph{real} elements
in $\sp(2,\C)$. Each quadric defines a \emph{non-constant} orthogonal
complex structure on
\[ S^4\sm S^1\ \cong\ S^2\times\C^+,\]
compatible with a scalar flat K\"ahler metric on the product
\cite{Pont,SV}. In each case, the discriminant locus is a circle
consisting of points in $S^4$ whose twistor fibres lie in $Z_i$. For
$Z_-,$ we have denoted this circle by $\Sm$.

An analogue of Proposition \ref{pD} can be proved in the same way:

\begin{proposition}\label{sFi}
  We have $Z_i=\pq^{-1}(\sF_i),$ where
\[ \sF_i = \big\{\buv\in\sM: \cuv=uv+2\u_i\v_i\big\}.\]
\end{proposition}

\noindent For example, $\sF_3$ has equation $uv =
\u_1\v_1+\u_2\v_2-\u_3\v_3.$ We shall use this in Section \ref{met}.
\medbreak

Since $\3$ acts diagonally on $\buv\in\sM,$ the $\3$-orbit containing
$\buv\in\sM$ is 2-dimensional (a 2-sphere) if and only if {\gr $\{\bfu,\bfv\}$
  is linearly dependent, i.e.\ $\buv$ belongs to $\sF_+\cup\sF_-$.} It follows
  that any $\3$-orbit in $\pq^{-1}(\sF_-)$ is a 2-sphere inside a quadric,
  though this orbit is a fibre of the twistor fibration over any point of
  $\Sm$. On the other hand, any $\3$ orbit in $\pq^{-1}(\sF_+\sm\sF_-$) is
  3-dimensional.

The restriction of the $\2$ 3-form $\varphi$ (recall Propositon
\ref{exact}) to a 3-dimensional $\3$ orbit $\sO$ in $\sC$ must be a
constant multiple of the volume form. The absence of 3-dimensional
cohomology in $\sC$ implies that this constant multiple must be
zero. Recalling \eqref{a+-} and \eqref{2ff}, set
\be{at} a = a_+ = 2|f_+|^2,\qquad t=\arg f_+.\ee
Since $u,v$ and $f_+=|f_+|e^{it}$ are constant on $\sO,$ it follows
that
\be{simple4} \varphi\we da\we dt\we du\we dv = 0.\ee
The proof of Proposition \ref{ball} tells us that the function $f_+/R$
factors through $\pi,$ so the same is true of $t$ and the
$SO(2)$-invariant function $a/R^2$. In fact, \eqref{sv} implies that
\be{aRs} 1-s^2=\frac{2a}{R^2},\ee
and we record without proof

\begin{lemma}
Let $X^\flat=X_*\ip s_4$ denote the 1-form dual to $X_*$ on
$S^4$. Then
\[ X^\flat = (1-s^2)dt.\]
\end{lemma}

\noindent One can regard $t\colon S^4\sm\Sf\to[0,2\pi)$ as a `longitude'
and $s\colon S^4\to[0,1]$ is sine of `latitude'.\smallbreak

Our aim is to replace the four functions $a,t,u,v$ in \eqref{simple4}
by three functions whose constancy defines a coassociative submanifold
$\sW$ of $\sC,$ so that $\sW$ has dimension $4$ and $\varphi$ pulls
back to zero on $\sW$. This is possible because each 3-dimensional
tangent space $T_o\sO$ is contained in a unique coassociative
subspace, which is generated by a basis $\{W_1,W_2,W_3\}$ of $T_o\sO$
and the 3-fold cross product $W_4$ defined by
\[ (\vasphi)(W_1,W_2,W_3,W)=\hh_2(W_4,W).\]
We shall in fact show that the desired functions are $a,t$ and
\be{b} b = u^2-v^2 = R(u-v).\ee
The value of $b$ is easily estimated at a point of $\sC$ for which
$f_-=0,$ for then
\[ R^2-\frac{b^2}{R^2} = 4uv = 4|f_+|^2+4|f_-|^2 = 4|f_+|^2=R^2(1-s^2),\]
and so $b=\pm sR^2$. Note that $b$ vanishes over $\Sm=\{s=0\}$.

Over any point $p$ of $S^4\sm\Sm,$ and for each value of $R>0,$ there
is a 2-sphere in $\sC$ lying over $p$ that has exactly two antipodal
`poles' belonging to the quadric $f_-=0$. Let
$[\bfz]=[\z_0,\z_1,\z_2,\z_3]$ be such a pole, chosen so that
$b>0$. Any point in the same $\R^3$ fibre with the same value of
$\|\bfz\|^2=R$ must (using \eqref{pi}) equal
$[\tz_0,\tz_1,\tz_2,\tz_3],$ where
\[\ba{rcl}
 \tz_0+j\tz_1 &=& (1+|\la|^2)^{-1/2}(\z_0+j\z_1)(1+j\la),\y
 \tz_2+j\tz_3 &=& (1+|\la|^2)^{-1/2}(\z_2+j\z_3)(1+j\la),\ea\] 
for some $\la\in\C\cup\{\infty\}$. Thus
\[  (1+|\la|^2)^{1/2}[\tz_0,\tz_1,\tz_2,\tz_3] =
[\z_0-\oz_1\la,\ \z_1+\oz_0\la,\ \z_2-\oz_3\la,\ \z_3+\oz_2\la].\]
The value of $b=R(u-v)$ is transformed into $\tilde b$ where
\[\ba{rcl} (1+|\la|^2)\tilde b 
&=& R\big(|\z_0\-\oz_1\la|^2-|\z_1\+\oz_0\la|^2
          +|\z_2\-\oz_3\la|^2-|\z_3\+\oz_2\la|^2\big)\y
&=& (1-|\la|^2)b - 
    2R(\la\oz_0\oz_1+\ol\la\z_0\z_1+\la\oz_2\oz_3+\ol\la\z_2\z_3)\y
&=& (1-|\la|^2)b,
\ea\]
since $[\bfz]$ satisfies $f_-=0,$ {\gr see Proposition \ref{pD}.} Now 
\[ h = \frac{1-|\la|^2}{1+|\la|^2}\in[-1,1]\] 
represents normalized height on the 2-sphere under stereographic
projection. Therefore
\be{tildeb} \tilde b = s\,h\,R^2,\ee
and the locus of points on a twistor fibre with $s>0$ and $\tilde b$
constant is a pole or a parallel circle, the equator if and only if
$b=0$. It follows from \eqref{match} that $R^2$ ($r^{1/2}$ in the
notation of \cite[\S4]{BS}) coincides with the norm of self-dual
2-forms, and is therefore the `natural' radial parameter of the
twistor fibres. Equation \eqref{tildeb} therefore tells us that, for
each fixed $s>0,$ the quantity $\tilde b$ is a Euclidean coordinate in
the fibre.

The equation $s=1$ distinguishes the 5-dimensional subset
$\sC_+=Q^{-1}(\sF_+)$ of $\sC$ lying over the totally geodesic
2-sphere $\Sf$. The twistor lift of $\Sf$ (in the sense of \cite{ES})
distinguishes a pole (where $f_-=0$) in each twistor fibre in
$\sC_+\cong\R^+\times Z_+$. It follows that the restriction of the
$\2$ form $\varphi$ to $\sC_+$ is a constant multiple of $d\tilde b
\we\tu_1,$ where $\tu_1$ is the self-dual 2-form determined by the
relevant point of $\Sf$. Setting $\tilde b$ constant therefore defines
a coassociative submanifold of $\sC_+$ which intersects each $\R^3$
fibre in a plane of constant height. The origin of this plane
corresponds to the pole of the twistor 2-sphere touching the
plane. The union (over $\Sf$) of these poles forms the unique
2-dimensional $\3$ orbit in the coassociative, which we can identify
with $TS^2$. If $\tilde b=0,$ the coassociative is a union of
equators, an example that was long recognized \cite{KM,KS}, though in
the present conical context this union is the orbifold $\C^2/\Z_2$
minus its singular point.

At this juncture, we can dispense with the tildes, and use $b$ to
denote the value of $u^2-v^2$ at an arbitrary point of $\sC$. We have
shown that $\sC$ contains a cosassociative submanifold with $s=1$ and
any constant value of $b\in\R$. Karigiannis and Lotay were the first
to identify the foliation of $\sC$ by coassociatives arising from the
action of $\3$ under consideration, and they extend the discussion to
the complete $\2$ metric on $\LaS$ \cite{KL2}. Our interpretation of
the conical situation using bivector and twistor space formalism is
summarized by

\begin{theorem}\label{cas}
Let $a\ge0,$ $b\in\R$ and $t\in[0,2\pi)$. Setting $a,\,b$ constant and
(if $s<1$) $t$ constant defines a coassociative submanifold of $\sC$
diffeomorphic to $TS^2$ unless $a=b=0$.
\end{theorem}

\noindent{\gr Note that if $s=1$ then $a=0$ and $t$ is undefined (see
  \eqref{at}), but $b$ can still be varied.}

\begin{proof}
We have motivated the significance of the $\3$ invariant functions
$a,b,t$ (see \eqref{at} and \eqref{b}), or equivalently $uv-\cuv,$
$u^2-v^2,$ $t$. The fact that these define coassociative manifolds
throughout $\sC$ follows from the differential relation
\be{imp} da\we db\we dt\ne0\qbox{provided} ab\ne0,\ee
and the identity
\be{simple3} \varphi\we da\we db\we dt = 0,\ee
which strenghtens \eqref{simple4}. We have verified both by
computer. Topologically, the situation is reminiscent
of \eqref{affine}, and we describe this next.

Fixing $t$ amounts to restricting attention to a totally geodesic
3-sphere
\[ S^3_t = (\R_t\op\R^3)\cap S^4,\]
where $\R_t\subset\R^2$ is the line corresponding to $\arg f_+=t,$ and
$SO(2)$ acts trivially on $\R^3$. The choice of $t$ will also
distinguish a point $p_t\in\Sm$. The configuration of the
coassociative manifolds above $S^3_t$ can be understood by reference
to the functions
\be{heights}\left\{\ba{ccl}
R^2 &=&\ds \frac{2a}{1-s^2}\\[15pt]
R^2\sqrt{1-h^2} &=&\ds
\frac{\sqrt{4a^2s^2-b^2(1-s^2)^2}}{s(1-s^2)},
\ea\right.\ee
formed by rearranging \eqref{aRs} and \eqref{tildeb}, with $a,b$
presumed constant. Whilst $R^2$ represents the natural radius of the
twistor 2-sphere containing the point in question, the second function
is the radius of the small circle generated by its $\3$ orbit.

The second equation in \eqref{heights} implies that $|h|=1$ when
$s=s_\mathrm{min},$ where
\be{smin} 1-s_\mathrm{min}^2 = 2c(\sqrt{c^2+1}-c)\ee
and $c=a/b$. The right-hand side lies in the interval $(0,1)$ provided
$a>0$ and $b\ne0$. In this case, the coassociative manifold projects
onto the semi-open annulus in $S^3_t$ defined by $s_\mathrm{min}\le
s<1$. It consists of points in $\sC$ whose norm squared $\|\bfz\|^2$
varies in inverse proportion to $\sqrt{1-s^2},$ and intersects each
$\R^3$ fibre over the annulus in a small circle that shrinks to a
point $p_\mathrm{min}$ over each point of the `limiting' 2-sphere
$\{s=s_\mathrm{min}\}\subset S^3_t$. Having fixed $a$ and $b,$ the
surface in Figure 2 depicts the union of small circles lying over a
geodesic segment in $S^4$ from a point of $\Sf$ to
$p_\mathrm{min}$. The associated value $R^2$ of the equatorial radii
are shown (in red) for reference. In this way, the surface represents
a fibre in the tangent bundle of this 2-sphere over its vertex
$p_\mathrm{min}$.

\begin{center}
\vspace{-10pt}
\scalebox{.35}{\includegraphics{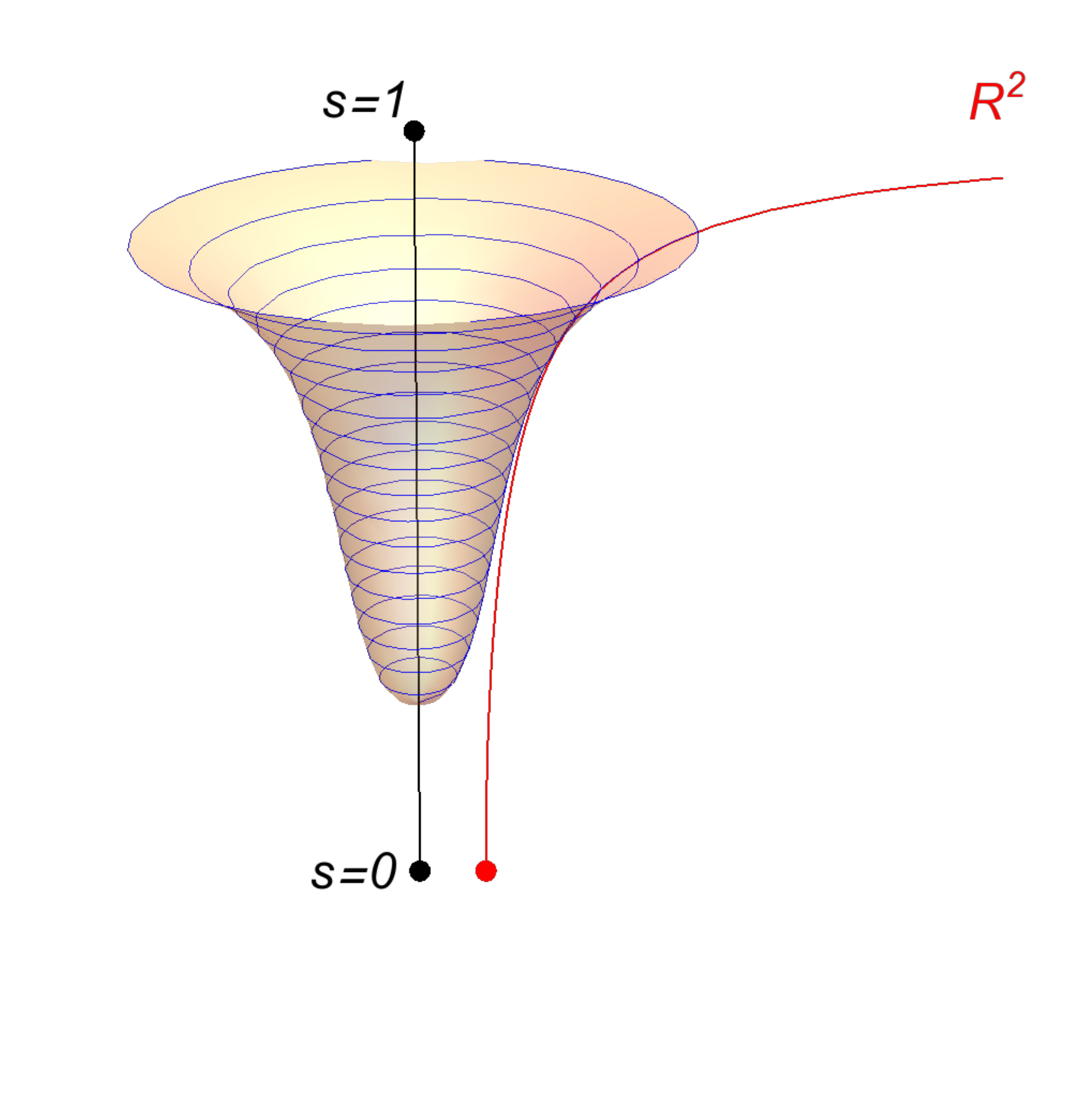}}
\vspace{-65pt}
\end{center}
\[\hbox{Figure 2: The fibre of a coassociative submanifold of $\sC$ with
$a=\half$ and $b=\qart$}\]\vs\vs

If $a>0$ but $b=0$ then \eqref{heights} implies that $h\equiv0,$ and
the coassociative is the closure of a union of equators in twistor
fibres assuming all values of $s\in(0,1)$. If Figure 2 were redrawn to
illustrate this case, the surface would retain a positive radius at
$s=0,$ though smoothness is maintained for reasons we now explain. For
a fixed value of $s$ close to $0$ the equators (of radius $R^2$ close
to $2a$) lie over a tiny 2-sphere close to $p_t$. As $s$ attains the
value $0,$ the limits of these equators exhaust the twistor fibre over
$p_t$ with $R^2=2a,$ which now plays the role of the limiting
2-sphere.

The case $a=0$ was discussed before the statement of the theorem.
\end{proof}

\setcounter{equation}0
\section{SU(3) structure}\label{symp}\label{SU3}

We now turn attention to the symplectic form
\[ \si = X\ip\varphi \]
obtained by contracting the exact $3$-form $\varphi$ on
$\sC=\R^+\ti\CP^3$ with the Killing vector field $X$ tangent to the
$SO(2)$ fibres. To proceed, one can use either of the descriptions
\[\ts \varphi = d(\frac13R^3\om) = -d(R\tu_0)\]
from Section \ref{holo}, provided we work over $\R^8$. We already know
that $\sL_X\om=0$ since $SO(2)$ is a symmetry of the nearly-K\"ahler
structure of $\CP^3$. But it is also true that
\[ \sL_X\tu_0 = \sL_X(dR\we\al_1 - \al_{23}) = 0.\]
This follows because $\sL_XY={\gr[X,Y]}=0$ and $\al_1=Y\ip e$ so $\sL_X\al_1=0,$ and
also $3\al_{23}=R^2(\om-\gr d\hal_1)$. Therefore
\[ \si = -X\ip d(R\tu_0) = d(R\,X\ip\tu_0).\]
We shall work from this formula, together with
 
\begin{lemma}\label{mual}
\[ X\ip\tu_0 = -u\,du+v\,dv+\half\big[\-u\,dv+v\,du + 
\bfu\cd d\bfv-\bfv\cd d\bfu\big].\] 
\end{lemma}

\begin{proof}
Recall that the function $\mu_i$ on $\sC$ was defined (in Section
\ref{T2}) to be the interior product $X\ip\al_i$. The definition of
$\tu_0$ therefore gives
\[ X\ip\tu_0 = -\mu_1dR-\mu_2\al_3+\mu_3\al_2 = -\mu_1dR+\half\Ga_-.\]
The result follows from Lemma \ref{Ga} after substituting $\mu_1=u-v$
and $dR=du+dv$.
\end{proof}

Bearing in mind that
$u\,du=\suml_{i=1}^3\!\u_id\u_i$ and $v\,dv=\suml_{i=1}^3\!\v_id\v_i,$
we can crudely approximate $X\ip\tu_0$ by the sum
\[\ts\half\suml_{i=1}^3(\u_i-\v_i)(d\u_i+d\v_i).\]
This leads us to define
\be{pq} \left\{\ba{rcl}
\bfp &=& \bfu+\bfv,\y \bfq &=& (u+v)(\bfu-\bfv)=R(\bfu-\bfv),
\ea\right.\ee
and write $\bfp=(p_1,p_2,p_3)$ and $\bfq=(q_1,q_2,q_3)$. (The context
should make it clear that these $q_i$ are not quaternions!) Then
\[\ba{rcl} \half\suml_{i=1}^3q_idp_i - R\,X\ip\tu_0 
&=& \frac32R(u\,du-v\,dv)+\half R(u\,dv-v\,du)\yy
&=& d(\half R^2\mu_1).\ea\] 
The fact that this 1-form is exact is somewhat of a miracle, since it
shows that the vectors $\bfp,\bfq$ furnish Darboux coordinates for
$\si$:

\begin{theorem}\label{si}
With the above notation,
\[ \si = -\half\sum_{i=1}^3 dp_i\we dq_i,\]
and this is non-degenerate away from the origin in $\R^6$.
\end{theorem}

The realization of the canonical coordinates in the theorem initially
came about by observing that setting $\bfp$ to be a constant vector
defines a Lagrangian submanifold of $\sM$. This assertion is
equivalent to the equation
\[ \si\we dp_1\we dp_2\we dp_3 = 0.\]
The idea of weighting sums and differences of the $\buv$ coordinates
with powers of the function $R=u+v$ arises from Proposition
\ref{ball}. The generic fibres of $\varpi\colon\sM\to D^3$ are
\emph{not} Lagrangian, but we do have:

\begin{corollary}\label{lag}
The following maps $\sM\to\R^3$ have $\si$-Lagrangian fibres:
\[\ba{ccc}
\buv\mapsto\bfp, && \buv\mapsto\bfq,\yy
\buv\mapsto \bfu\sqrt{u\+v}, &\quad& \buv\mapsto\bfv\sqrt{u\+v}.
\ea\]
\end{corollary}

\begin{proof}
The Lagrangian nature of the first two follows immediately from Theorem
\ref{si}. By (skew) symmetry it suffices to verify the third, so set
$\bfu=R^{-1/2}\bfm$ with (as always) $R=u+v$ and $\bfm$ a constant vector. The
restriction of $\gr 2\si$ to a fibre is
\be{2si}\ts\gr \suml_{i=1}^3\big(\frac12R^{-3/2}m_i\,dR-dv_i\big)\we
(\frac12R^{-1/2}m_i\,dR-R\,dv_i-v_i\,dR) = 
\Big(\suml_{i=1}^3v_idv_i\Big)\we dR.\ee
But 
\[ dR = du+dv = -\half R^{-3/2}|\bfm|dR+dv\]
and $\sum\!v_idv_i=vdv,$ so \eqref{2si} vanishes.
\end{proof}

The methods adopted at the start of this section are suited to a study
of the symplectic form
\[ \BSsi = -X\ip d\big((R^4+1)^{1/4}\tu_0\big) =
d\big((R^4+1)^{1/4}\,X\ip\tu_0\big)\] 
induced from the complete $\2$ structure of Theorem \ref{BS}. An
initial observation is that the three subspaces defined by the
respective equations $\bfu=\bf0,$ $\bfv=\bf0,$ $\bfu=\bfv$ are
Lagrangian relative to $\BSsi$ (as they are for $\si$). This fact
follows by substituting the equations into Lemma \ref{mual} and then
differentiating, and is equally apparent from

\begin{corollary} The 2-form $2(R^4\+1)^{3/4}\BSsi$ equals
\[ -(R^4+2)du\we dv + R^3dR\we(\bfu\cd d\bfv\-\bfv\cd d\bfu) +
2(R^4\+1)\suml_{i=1}^3 du_i\we dv_i.\]
\end{corollary}

\begin{proof}
We again use Lemma \ref{mual} and the same calculations that led to
Theorem \ref{si}. This shows that
\[\ts 2(R^4\+1)^{3/4}\BSsi = R^3dR\we\Ga_- + 
2(R^4+1)\big(\!-d\mu_1\we dR+du\we dv+\suml_{i=1}^3du_i\we dv_i\big),\]
which simplifies to the expression stated. 
\end{proof}

For the remainder of this section, we shall consider exclusively the
metric $\gg_2$ described by Theorem \ref{g2}. We shall study the
almost complex structure $\JJ$ on $\sM$ defined by
\be{siNg} \si(W,Y) = N_2^{1/2}\gg_2(\JJ W,\,Y),\ee
in accordance with \eqref{s} and Remark \ref{scale}. Note that both
$\gg_2$ and (as must be the case) $\JJ$ are unaffected by re-scaling
the Killing vector field $X$ used in their definition. We shall also
identify the complex volume form induced from the $\2$ structure of
$\sC$.

Consider a tangent vector $\E$ in $T_m\sM,$ it is natural to define
dual 1-forms
\[ \E^\flat=\E\ip\gg_2,\qquad \E^\natural=\E\ip\si.\]
With the `endomorphism' sign convention for the action of $\JJ$ on
1-forms, \eqref{siNg} becomes
\be{nat} \E^\natural = N_2^{1/2}\,\JJ\E^\flat.\ee
This suggests the following strategy to try to pin down $\JJ$. We seek
a tangent vector $\E$ such that \emph{either} $E^\flat$ \emph{or}
$\E^\natural$ is as simple as possible, in the hope that the other one
is not over complicated.

We give one example of this approach. For this purpose, let $\xi$
denote the 1-form $\frac12\sum\!q_j\,dp_j$ (cf.\ \eqref{pq}) and let
$B_i=u\,d\v_i-v\,d\u_i$ as at the end of Section \ref{T2}. Then

\begin{proposition}\label{10}
For each fixed $i=1,2,3,$
\[\ts N^{1/2}\JJ(d(Rp_i)) = 2R^{-1}p_i\xi + 2R\,B_i- q_i\,dR.\] 
\end{proposition}

\begin{proof} 
For clarity of notation, we set $i=3$ and define
\[  \E = u\frac\pd{\pd u_3}+v\frac\pd{\pd v_3}.\]
This belongs to the annihilator of the 1-form $\bfB=(B_1,B_2,B_3),$ and so has
zero contraction with both 1-forms $\Ga_+,\Ga_-$ {\gr defined by \eqref{Ga+-}.}
Theorem \ref{g2} then implies that
\[\ts \E^\flat = 
\frac12R(d\u_3+d\v_3) +\frac12(\u_3+\v_3)dR = \frac12d(Rp_3).\]
On the other hand,
\[\ts \E^\natural = 
-\half(\u_3\-\v_3)R\,dR +
\half(\u_3\+\v_3)\suml_{j=1}^3(u_j\-v_j)(d\u_j+d\v_j) + RB_3.\]
The result follows from \eqref{nat}.
\end{proof} 

From the discussion in Section \ref{GH}, we have
\[ N^{-1/4}\psi^+ = X\ip(\vasphi)= d\beta,\]
where $\beta=-\qart R^4\,X\ip\Re\Up$. Expressing $\beta$ in terms of
$\buv$ is not hard:\vs

\begin{lemma}\label{beta}
The 2-form $\beta$ is given by 
\[ 16R^{-1}\,\beta\ \ceq\
3\T\bfu,d\bfv,d\bfv, - uv^{-1}\T\bfv,d\bfv,d\bfv,+
4\T\bfu,d\bfu,d\bfv,.\]
\end{lemma}\vs

To find $\psi^-$ and the $(3,0)$-form $\Psi=\psi^++i\psi^-,$ one needs
to involve $\Th_2$ more directly. This is achieved in the next result,
which we quote without proof. It illustrates the complexity of the
$SU(3)$ structure induced on $\sM$.

\begin{theorem}\label{Psi}
The space of $(3,0)$ forms on $\sM$ is generated by
$\Psi=\psi^++i\psi^-,$ where
\[\ba{rcl}
8uv\,\psi^+\kern-10pt &\ceq& 
\frac16v(N_2+4v^2)\T d\bfu,d\bfu,d\bfu, 
- v(4u^2+3uv+\cuv)\T d\bfv,d\bfu,d\bfu,\y
&&\kern-15pt 
+\big((u+2v)\bfv\cd d\bfv + v\bfu\cd d\bfv\big)\we\T\bfu,d\bfu,d\bfu,  
+(v\bfu\cd d\bfv - u\bfv\cd d\bfv)\we\T\bfv,d\bfu,d\bfu,.\yyy

4N_2^{1/2}\psi^-\kern-10pt &\ceqq&
\frac13(N_2+4v^2)\T d\bfu,d\bfu,d\bfu, 
+ \big((3+uv^{-1})\cuv -3u^2-5uv\big)\T d\bfu,d\bfv,d\bfv,\y 
&&\kern-15pt 
+ 2\,\bfv\cd d\bfv\we\T\bfu,d\bfu,d\bfu,
+ 2\,uv^{-1}\bfv\cd d\bfu\we\T\bfv,d\bfv,d\bfv,\y
&&\kern-15pt
+ \big((1-uv^{-1})\bfv\cd d\bfv+(3+vu^{-1}\big)\bfu\cd d\bfv)
\we\T\bfv,d\bfu,d\bfu,.
\ea\]
\end{theorem}

\noindent We have verified by computer and Proposition \ref{10} that the
$(1,0)$-form
\[\gr \ga = (d(Rp_i))^{1,0} = (1-i\JJ)d(Rp_i)\]
satisfies $\ga\we\Psi=0$.\smallbreak

\begin{corollary}\label{non}
The almost complex structure $\JJ$ induced on $\sM$ is not integrable.
\end{corollary}

\begin{proof}
From Remark \ref{scale}, one must verify that
$d(N_2^{-1/4}\psi^-)\ne0$. In fact, to do this, we have computed
\[  d(N_2^{-1/4}\psi^-)\we d\u_1\we d\u_2\>\big|_{(2,2,1;2,2,1)} = 
{\ts\frac23}d\u_{123}\we d\v_{123},\] 
where the left-hand side has been evaluated at $\bfu=(1,2,2)$ and
$\bfv=(1,2,2),$ so that (conveniently) $u=v=3$.
\end{proof}

The next result implies that $\psi^+$ vanishes on the 3-dimensional
subspace $\left<d\u_i,d\v_i,dR\right>^\mathrm{o}$ of $T_m\sM,$ for any
$m$ and fixed $i$:

\begin{lemma}\label{gen}
\[ \psi^+\we d\u_i\we d\v_i\we dR = 0,\qquad i=1,2,3.\]
\end{lemma}

\begin{proof}
To check the equation, take $i=3$ again. The exterior product of
$d\u_3\we d\v_3$ with the \emph{visible part} of $8uv\psi^+$ displayed
in Theorem \ref{Psi} equals
\[\ba{l}
 u(\u_3\v_1-\v_1\v_3)+v(\u_1\u_3+\u_1\v_3+2\u_3\v_1)d\v_{31}\we d\u_{123}\y 
\hskip30pt+
u(\u_3\v_2-\v_2\v_3)+v(\u_2\u_3+\u_2\v_3+2\u_3\v_2)d\v_{32}\we d\u_{123}.\ea\]
Wedging further with $dR$ yields
\[ (\u_1\u_3\v_2+\u_1\v_2\v_3-\u_2\u_3\v_1-\u_2\v_1\v_3)d\v_{123}\we d\u_{123},\]
but this is cancelled by the symmetrization implicit in the relation
$\ceq\!\!$.
\end{proof}

In contrast to the lemma, one can verify that
\[ \psi^-\we d\u_3\we d\v_3\we dR\ne0,\]
except at points where $\u_1=\u_2=\v_1=\v_2=0$ or $\u_3=\v_3=0$. It
follows that the subspace $\left<d\u_i,d\v_i,dR\right>$ of $T^*_m\sM$
admits no $\JJ$-invariant 2-plane for generic $m,$ though the key word
is `generic'. For if we restrict the $(3,0)$ form $\Psi$ of Theorem
\ref{Psi} to any 4-dimensional linear subvariety $\sM(\bfn)$ (recall
Definition \ref{sMn} and Corollary \ref{van}), the result is zero
because all the triple products vanish. This implies\bigbreak

\begin{theorem}\label{Jholo}
$\sM(\bfn)$ is $\JJ$-holomorphic for each $\bfn\in S^2$.
\end{theorem} 

\begin{proof}
  We present an argument that does not depend on the calculation of
  $\psi^\pm$. Take $\bfn=(1,0,0)$ for definiteness, so that
  $\sM(\bfn)$ lies in
\[ \R^4 = \{\buv\in\R^6:\u_1=0=\v_1\}.\]
Consider the tangent vectors 
\[\pd_1=\pd/\pd u_1,\qquad \pd_4=\pd/\pd v_1\]
in $\R^6$ defined along this $\R^4$ (they are normal to the $\R^4$
relative to the flat metric.) It follows Theorem \ref{si} that
\be{RR}\partial_1{}^\natural = R\,d\v_1,\qquad
   \partial_4{}^\natural = -R\,d\u_1.\ee
A key reason for this simplicity is that
\[ \pd_1\ip du = u^{-1}\pd_1\ip(\u_1d\u_1),\]
which vanishes along $\R^4,$ similarly for $\pd_4\ip dv$ and interior
products with $dR=du+dv$. Theorem \ref{g2} implies that
\be{NN}\ba{rcl}     
\partial_1{}^\flat &=&\half N^{-1}\big[(N+4v^2)d\u_1+(N-4uv)d\v_1\big]\y
\partial_4{}^\flat &=&\half N^{-1}\big[(N-4uv)d\u_1 + (N+4u^2)d\v_1\big],
\ea\ee 
where $N=N_2=6uv-2\cuv$. This time, a key point is that the interior
products with $\Ga_+$ and $\Ga_-$ vanish along $\R^4$ (as in the proof
of Proposition \ref{10}). It follows from \eqref{nat} that
$\JJ\,d\u_1$ and $\JJ\,d\v_1$ both belong to $\left<du_1,dv_1\right>$
at all points of $\R^4$ for which $uv\ne0$. The {\gr annihilators} of these
subspaces are the tangent spaces to $\R^4,$ and are therefore
$\JJ$-invariant.
\end{proof}

\begin{remark}\rm
We know from \eqref{nat} that $\JJ\E^\natural=-N^{1/2}E^\flat$ for any
tangent vector $\E$. A computation of $\JJ^2$ involves the determinant
\[ (N+4u^2)(N+4v^2) - (N-4uv)^2 = 4R^2N,\]
and allows us to verify that $\JJ^2=-\bf1$. This confirms that the
symplectic form $\si$ is correctly normalized in Theorem \ref{si}. We
can also strengthen Corollary \ref{non} using the $(1,0)$-forms
$\vep_i=(1-i\JJ)d(Rp_i)$ made explicit in Proposition \ref{10}. A
computation shows that
\[ d\vep_2\we\vep_2\we\vep_3\we d\u_1\we d\v_1 \ne 0
\qbox{along} \u_1=0=\v_1,\]
so the restriction of $\JJ$ to $\sM(\bfn)$ is not integrable. Note that $\JJ$
degenerates across the locus $N=0,$ i.e.\ when $u=0$ or $v=0$.
\end{remark}

Finally, the vanishing of both $\Psi$ and $F_2$ when restricted to $\sM(\bfn)$
{\gr (see Corollary \ref{van}) is consistent with the first equation in Lemma
  \ref{FF}.}

\setcounter{equation}0
\section{Metrics on subvarieties}\label{met}

We return to consider the family of Riemannian metrics $\gg_c$ {\gr introduced}
by Definition \ref{ggc} and described by Lemma \ref{gc}. Restricted to the
hypersurface $R=1,$ each is the pushdown of the metric $\wh_c$ described by
Proposition \ref{hc}. Recall that $\wh_1$ is the K\"ahler metric on $\CP^3$ and
$\hh_2$ has holonomy $\2$ on the cone $\sC=\R^+\ti\CP^3$.  The diagonal action
of $\3$ on
\[\{\buv\in\R^6: \bfu,\bfv\in\R^3\}\]
defines an isometry of $(\sM,\gg_c)$ for any $c<3$.

We first present a result that motivated other results in this
section.

\begin{proposition}\label{1ff}
The restriction of $\gg_c$ to the negative quadrant
\[ \{\buv=(u,0,0;\,-v,0,0):u,v>0\}\subset\sF_-\]
in $\sM$ has first fundamental form
\[ \Big(1+\frac v{2u}\Big)du^2+du\,dv+\Big(1+\frac u{2v}\Big)dv^2\]
independently of $c<3,$ and zero Gaussian curvature.
\end{proposition}

\begin{proof}
This result may be regarded as a corollary of Propositions \ref{hc} and
\ref{th}. Up to the action of $SO(2),$ we can realize the subset
$\u_2=\u_3=\v_2=\v_3=0$ of $\sM$ by taking 
\[ \x_2=\x_3=\x_4=\x_5=0,\qbox{i.e.}\z_1=\z_2=0,\]
so that
\[ u=x_0^2+x_1^2=|\z_0|^2,\quad v={\gr x_6^2+x_7^2}=|\z_3|^2.\]
Moreover, $\al_2=\al_3=0$ so $\mu_2=\mu_3=0,$ and $\mu_1=u-v$. Thus $\hh_c$ and
$\Th_c$ are independent of $c$. {\gr Proposition \ref{th} implies that}
$N_c=\hh_c(X,X)$ is independent of $c$ on $\sF_-$. The first fundamental form
can now be read off from Theorem \ref{g1}: $\gg_1$ restricts to
\[ ds^2 = \half(du+dv)^2 + {\gr\half R}(u^{-1}du^2+v^{-1}dv^2),\]
which simplifies to that stated.

To prove that this metric has zero Gaussian curvature where defined,
we use the substitution
\be{Rcs}\left\{\ba{l} 
u = R\cos^2(\phi/2)\y v = R\sin^2(\phi/2), \ea\right.\ee
with $0\le\phi\le\pi$. Then
\be{dudv}\left\{\ba{l} 
du = \cos^2(\phi/2)dR - R\cos(\phi/2)\sin(\phi/2)d\phi\y 
dv = \sin^2(\phi/2)dR + R\cos(\phi/2)\sin(\phi/2)d\phi.
\ea\right.\ee 
A straightforward calculation reveals that
\be{rev} ds^2 = dR^2+\half R^2d\phi^2,\ee
which is the metric on a double cone in $\R^3$ with half
angle $\pi/4,$ and certainly flat.
\end{proof}

\begin{remark}\label{super}\rm
Because of the $\3$ invariance, $\gg_c$ will have an identical nature
on any negative `diagonally linear' quadrant in $\R^6,$ and one can
also switch the minus sign from $\v_1$ to $\u_1$. It follows from the
proof of Proposition \ref{1ff} that such a quadrant lies in the
projection of the cone over a complex projective line $\CP^1\subset
Z_-$ tangent to the horizontal distribution $D$ (see Corollary
\ref{hor}), and can therefore said to be \emph{superminimal}
\cite{Br3}. The range of the angle $\phi$ in the proof was restricted
to $(0,\pi),$ but can be extended to $\R/(2\pi\Z)$ to include the
image of $j(\CP^1)$. It also follows from the proof of the proposition
that $\gg_1$ takes an identical form in the positive quadrant
\[ \{\buv=(u,0,0;\,v,0,0):u,v>0\}\subset\sF_+, \]
though $\gg_c$ will have a slightly different (albeit, flat) form if
$c\ne1$.
\end{remark}

We shall use the substitution \eqref{Rcs} throughout this section, in
order to revert to a conical description
\be{wgc} \gg_c = dR^2 + R^2\wg_c\ee
of the induced metrics on $\sM,$ reflecting their origin in $\sC$.
There are three justifications for squaring the trigonometric
functions: (i) it ensures that $R=u+v$ is as before, (ii) it amounts
to using polar coordinates for the complex moduli $|\z_i|$ in $\sC,$
and (iii) it simplifies the form of $\gg_c$ in subsequent
statements. The choice of the half-angle is less significant.

In addition, we set
\[\left\{\ba{rcl}
 \bfu &=& u\,\bfs,\y
 \bfv &=& v\,\bft,
\ea\right.\]
so that $\bfs,\bft\in S^2,$ and $|d\bfs|^2=d\bfs\cd d\bfs$ and 
$|d\bft|^2=d\bft\cd d\bft$ denote the round metrics. Then
\[\left\{\ba{rcl}
 d\bfu &=& \bfs\,du+u\,d\bfs,\y
 d\bfv &=& \bft\,dv+v\,d\bft.
\ea\right.\]
{\gr We also set $\bfs\cd\bft=\cos2\th,$ so that the angle between $\bfu$ and
$\bfv$ equals $2\th$.}

Using the methods of Section \ref{T2}, one can show that $\wg_c =
\wg_c(\phi,\gr\bfs,\bft)$ extends to a smooth bilinear form on $\R^6,$ though we
shall restrict our discussion to the cases $c=1$ and $c=2$. Lemma \ref{gc}
suggests that the second case faithfully reflects the behaviour of $\gg_c$ for
all values of the parameter $c$ with $1<c<3$. In an attempt to identify the two
2-spheres in $\R^3\cup\R^3$ (and thereby bypass singularities), we shall first
describe the restriction of $\gg_1$ and $\gg_2$ to the subvarieties \be{coord}
\sF_\pm\ \cong\ \R^+\times[0,\pi]\times S^2.\ee of Definition \ref{align}. The
bijection \eqref{coord} is determined by the coordinates $(R,\phi,\bfs)$.

Theorem \ref{g1} implies that $\gg_1$ is the sum of $\half R$ times
$u|d\bfs|^2+v|d\bft|^2$ and the first fundamental form \eqref{rev}, before
restricting to $\sF_\pm,$. It follows that
\be{g1bis} {\gr 2\,\wg_1 =} \cos^2(\phi/2)\big|d\bfs\big|^2 +
\sin^2(\phi/2)\big|d\bft\big|^2 + d\phi^2.\ee
{\gr Since $\sF_\pm$ corresponds to setting $\bft=\pm\bfs,$} the restriction of
$\gg_1$ to \emph{both} $\sF_+$ and $\sF_-$ equals
\[    dR^2 + \half R^2(|d\bfs|^2 + d\phi^2).\]
By contrast,

\begin{corollary}\label{sF+-}
The restriction of $\gg_2$ to $\sF_\pm$ equals $dR^2+R^2\wg_2$ where
\[\wg_2 = \left\{\ba{ll} 
\half|d\bfs|^2+\qart d\phi^2 & \hbox{on } \sF_+\yy
\frac18(3+\cos2\phi)|d\bfs|^2+\half d\phi^2 &
\hbox{on } \sF_-.\ea\right.\]
\end{corollary}

\begin{proof}
We shall examine the various terms for $\gg_2$ in Theorem \ref{g2}. With
  the assumption that $\buv\in\sF_\pm,$ we can write $\bfu=u\bfs$ and
  $\bfv=\pm\bfs,$ and
\[\ba{rcl} |d\bfu+d\bfv|^2 
&=& \big|(du\pm dv)\bfs+(u\pm v)d\bfs\big|^2\yy
&=& \left\{\ba{ll} 
dR^2+R^2|d\bfs|^2 &\hbox{on }\sF_+\y
(du-dv)^2 + (u-v)^2|d\bfs|^2&\hbox{on }\sF_-.
\ea\right.\ea\]
Moreover,
\[\ba{rcl} |\bfB|^2 
&=& \big|\pm u(dv\,\bfs+v\,d\bfs)-v(du\,\bfs+u\,d\bfs)\big|^2\yy
&=& \left\{\ba{ll}
(u\,dv-v\,du)^2 = R^2uv\,d\phi^2 &\hbox{on }\sF_+\y
(u\,dv+v\,du)^2+4u^2v^2|d\bfs|^2 &\hbox{on }\sF_-.\ea\right.
\ea\]
Now $N=N_2=6uv-2\cuv$ simplifies to $4uv$ on $\sF_+$ and $8uv$ on $\sF_-$.
{\gr Definition \ref{align} and \eqref{Ga+-} imply that $\Ga_+=0$ on $\sF_+$
  and $\Ga_-=0$ on $\sF_-$; thus}
\[\gr\ba{rcl} \Ga_+^2
&=& \left\{\ba{ll}
0 &\hbox{on }\sF_+\\
4(u\,dv+v\,du)^2 &\hbox{on }\sF_-,\ea\right.\\[13pt]
  \Ga_-^2
&=& \left\{\ba{ll}
4(u\,dv-v\,du)^2 &\hbox{on }\sF_+\\
0 &\hbox{on }\sF_-.\ea\right.\ea\]
The proof is completed by adding up the various terms.

The case of $\sF_+$ is easiest, and can also be deduced from Lemma
\ref{gc}. {\gr For $\sF_-,$ the coefficient of $|d\bfs|^2$ in $\gg_2$ equals
\[\ts \half(u^2+v^2) = \frac18(3+\cos2\phi).\]
The terms involving $|\bfB|^2$ and $|\Ga_+|^2$ also contribute
\[ \frac1{2uv}(u\,dv+v\,du)^2 = \half R^2d\phi^2+2du\,dv,\]
and the last term above converts $\half(du-dv)^2$ into the missing $\half
dR^2$.}
\end{proof}

\begin{remark}\rm
The restriction of $\gg_1$ to $\sF_\pm$ is invariant by diagonal translation in
the $(u,v)$ plane defined by rotating the angle $\phi$. This action is not
however an isometry for $\gg_2|_{\sF_-}$.  Restricted to four dimensions,
neither metric degenerates where only one of $u,v$ is zero. Equation
\eqref{g1bis} shows that $\gg_1$ is singular on each locus $\{u=0\}$ and
$\{v=0\},$ since the respective 2-sphere is shrunk to a point. The same is true
for $\gg_2$ because
\be{sing} \ba{rcl}
\lim\limits_{\gr u\to0}\wg_2 &=& |d\bfs|^2+\qart(3-\cos2\th)d\phi^2,\yy
\lim\limits_{\gr v\to0}\wg_2 &=& |d\bft|^2+\qart(3-\cos2\th)d\phi^2,
\ea\ee
limits that are verified using Theorem \ref{g2}. Although $\gg_1$ has little to
do with $\2$ holonomy, it \emph{is} associated to a $\2$ structure with closed
3-form that arises as the $\1$ quotient of the flat $\Spin(7)$-structure on
$\R^8$ \cite{Fow}.
\end{remark}

The 4-dimensional subvarieties $\sF_\pm$ of $\sM$ are distinguished by
their $\3$-invariance. Neither can be $\JJ$-holomorphic; this follows
from

\begin{lemma}\label{S2uv}
Fix $\sF_+$ or $\sF_-$. Each 2-sphere $S^2_{u,v}\subset\sF_\pm$ defined by
setting $u$ and $v$ equal to positive constants is totally real in $\sM$;
indeed $\JJ(TS^2_{u,v})$ is {\gr $\gg_2$-orthogonal} to $T\sF_\pm$.
\end{lemma}

\begin{proof}
Let $V,W$ be vectors tangent to $\sF_\pm$ with $V$ tangent to one of
the 2-spheres $S^2_{u,v}$ at some point. {\gr Using Lemma \ref{mual},
pullback of $\si$ ro $\sF_\pm$ is given by}
\[\gr \si = \left\{\ba{rl}
R\,du\we dv &\hbox{on }\sF_+\yy
-2R\,du\we dv &\hbox{on }\sF_-.
\ea\right.\] 
Therefore
\[ {\gr\gg_2}(\JJ V,W)=\si(V,W)=0,\]
as asserted.
\end{proof}

From \eqref{coord}, we see that $\gr S^2_{u,v}$ parametrizes a family of
real surfaces inside $\sF_\pm$. Up to the $\3$ action, each is
equivalent to one of the quadrants of Proposition \ref{1ff} or
Remark \ref{super}. By Corollary \ref{sF+-}, the tangent spaces of the
leaf
\[ \{(u\bfs,\,\pm v\bfs):u,v>0\}\subset\sF_\pm\]
are {\gr $\gg_2$-orthogonal} to the distribution $\gr TS^2_{u,v}$ in
$T\sF_\pm,$ so Lemma \ref{S2uv} implies that the leaf is
$\JJ$-holomorphic. This is also a corollary of Theorem \ref{Jholo}, since the
leaf is the intersection of two of the $\JJ$-holomorphic surfaces.

We began this section by restricting $\gg_c$ to a 2-dimensional subspace $\R^2$
of $\R^6,$ and then extended that result to the 4-dimensional subvarieties
$\sF_+,\sF_-$. It is natural to consider too the subvarieties highlighted by
Definition \ref{sMn} and Theorem \ref{Jholo}. For definiteness, we shall take
$\bfn=\gr(1,0,0)$ this time, so that
\be{sM3} \sM(\bfn)=\{\buv\in\sM:\u_1=0=\v_1,\ uv\ne0\}.\ee
Since $\bfn$ will remain fixed for the remainder of this section, we shall
denote the subspace of $\R^6$ containing \eqref{sM3} merely by $\R^4$.

The choice of setting $\u_1,\v_1$ to zero (contrasting with that of
Proposition \ref{1ff}) is dictated by their expression as quadratic
forms diagonalized by $\x_0,\ldots,\x_7,$ so that
$|\z_0|=|\z_2|,$ and $|\z_1|=|\z_3|$.
This enables us to take
\[ [\bfz]=[\z_0,\z_1,\z_2,\z_3] = 
   [\la e^{i\al},\mu e^{i\beta},\la e^{i\ga},\mu e^{i\de}],\]
with $\la,\mu\ge0$ and $\la\mu\ne0$. (Square brackets again represent
the $\1_1$ quotient.) Then
\[\left\{\ba{rcl}
 \bfu &=& \big(0,\>u\cos(\th+\ch),-u\sin(\th+\ch)\big),\y
 \bfv &=& \big(0,\>v\cos(\th-\ch),\ \>v\sin(\th-\ch)\big),
\ea\right.\]
where $u=2\la^2,$ $v=2\mu^2$ and $\al-\ga=\th+\ch$ and
$\beta-\de=\th-\ch$. Since $\cuv=uv\cos2\th,$ the angle
between $\bfu$ and $\bfv$ is again $\th$ and is independent of $\ch$.

Using the coordinates $(R,\phi)$ and substituting expressions for
$\bfu,\bfv$ into $\gg_2$ yields

\begin{theorem}\label{R4}
The restriction of $\wg_2$ to $\R^4$ equals
\[\ba{l}
\half d\th^2+\frac18(3-\cos2\th)d\phi^2
+\frac1{16}(7+\cos2\th+2\sin^2\th\,\cos2\phi)d\ch^2\y
\hskip180pt 
+\cos\phi\,d\th\,d\ch-\qart\sin2\th\sin\phi\,d\phi\,d\ch.\ea\]
\end{theorem}

\noindent{\gr Note that on the locus $\phi=0$ ($v=0$) or on the locus $\phi=\pi$
($u=0$), three of the terms above sum to a perfect square, and
\[\ts \wg_2 = \half(d\th\pm d\phi)^2+\frac18(3-\cos2\th)d\phi^2\]
is degenerate.}\smallbreak

We now define subsets
\[\ba{ccl} 
\sF_-'&=&\{(0,-u\sin\ch,-u\cos\ch;\,0,\ v\sin\ch,\ v\cos\ch)\}\y
\sF_1'&=&\{(0,\ u\cos\ch,\ -u\sin\ch;\ 0,\ v\cos\ch,-v\sin\ch)\},\ea\]
corresponding to to $\th=\pi/2$ and $\th=0,$ and 
\[\ba{ccl} 
\sF_2'&=&\{(0,-u\sin\th,-u\cos\th;\ 0,\ v\sin\th,-v\cos\th)\}\y
\sF_3'&=&\{(0,\ u\cos\th,\ -u\sin\th;\ 0,\ v\cos\th,\ v\sin\th)\},
\ea\] corresponding to $\ch=\pi/2$ and $\ch=0$. In all cases,
$u,v\ge0$ and $0\le\th,\ch\le2\pi$. It follows from the definition in
Proposition \ref{sFi} that
\[ \sF_-'=\sF_-\cap\R^4,\qbox{and} 
\sF_i'=\sF_i\cap\R^4\quad\forall i=1,2,3,\]
to which we can add that $\sF_1'=\sF_+\cap\R^4,$ since the equations
of $Z_+$ and $Z_1$ coincide when $|\z_1|=|\z_3|$. But there is also a
duality between the pairs $\{{\gr\sF_-'},\sF_1'\}$ and $\{\sF_2',\sF_3'\}$
that derives from Remark \ref{ambi}: changing the sign of $\v_3$ swaps
the subsets over, though this map is definitely not
$\3$-equivariant.

Observe that $\sF_2'=R_{23}(\sF_3')$ where $R_{23}\in SO(2)$ is rotation by
$\pi/2$ in the `2--3' plane. Any non-zero vector in $\R^6$ lies in the $\3$
orbit of a point of $\sF_2'$ or $\sF_3'$ (this orbit is a 2-sphere if $\th$ is
a multiple of $\pi/2$). Therefore both $\sF_2'$ and $\sF_3'$ are \emph{slices}
for the $\3$ action, corresponding to the invariants $u,v,\th$. These isometric
slices are characterized by an almost symmetric treatment of $\bfu$ and $\bfv,$
which leads to a simpler form of the induced metric below. The intersection
$\sF_2'\cap\sF_3'$ contains the union $\R^2\cup\R^2$ of the two 2-planes given
by $uv=0$.

Theorem \ref{R4} and the subsequent discussion yields

\begin{theorem}
An open subset of $\sM=\R^6\smz$ is foliated by 3-dimensional
submanifolds parametrized by $\3,$ each member of which is a cone over
a cylinder isometric to $\sF_2'$ (or $\sF_3'$) endowed with the metric
$dR^2+R^2\wg_2$ where
\[\ts \wg_2 = \half d\th^2+\frac18(3-\cos2\th)d\phi^2.\]
\end{theorem}

\vs

\begin{remark}\rm
Since the quadrics $Z_1,Z_2,Z_3$ are equivalent under $SU(2),$ their
images in $\sM$ are congruent under $\3$ and thus all isometric. What
is more surprising is that $\sF_2$ embeds isometrically into $\sF_-$
after a change of coordinates $\th\mapsto\phi+\frac\pi2$. At first
sight this appears to contradict Corollary \ref{sF+-}, given that
$\sF_1$ coincides with the \emph{positive} space $\sF_+$ over
$\R^4$. The problem is that, although $\sF_1=R_{13}(\sF_3),$ this
relation no longer holds if we apply primes.
\end{remark}

Let us try to visualize $\sF_2'$ and $\sF_3'$. Each is a cone over a
cylinder $S^1\times[0,\pi],$ with $\th\in S^1$ and $\phi\in[0,\pi]$.
If one sets
\[\ts f(\th)=\frac12\sqrt{3-\cos2\th},\]
then $2\wg_2$ is the first fundamental form of the surface of
revolution in $\R^3$ parametrized by
\[ \big(f(\th)\cos\phi,\>f(\th)\sin\phi,\>g(\th)\big).\] 
Here $f'(\th)^2+g'(\th)^2=1,$ so that the profile curve
$(f(\th),g(\th))$ has unit speed. Its Gaussian curvature
\[ K= -\frac{f''(\th)}{f(\th)}= 1-\frac8{(3-\cos2\th)^2}\]
varies between $-1$ and $\frac12$.

Figure 3 represents $(\sF_2'\cup\sF_3')/\R^+$ topologically as a 2-torus,
obtained by identifying the outer boundaries of the curvilinear rectangle in
the usual way. We have chosen to represent $\gr\sF_3'$ by the blue patch, whose
boundaries correspond to $\phi=\pi,0$ (meridians left and right forming the
intersection $\sF_2'\cap\sF_3'$) and to $\th=0,2\pi$ (semicircles bottom and
top that are identified to form the cylinder). Attaching $\sF_2'$ (the yellow
patch) to $\sF_3'$ requires a vertical jump $\th\mapsto\th+\pi/2$
(corresponding to $L\in SO(2)$); this is shown schematically for $u=0$ but the
combination is not smooth (the parallels are not $C^2$ in $\phi$). The isometry
$\hj$ of Lemma~\ref{j} acts by interchanging $\th\leftrightarrow2\pi\-\th$ and
$\phi\leftrightarrow\pi\-\phi,$ and therefore flips the upper and lower halves
of each coloured surface.

\begin{center}
\vspace{15pt}
\scalebox{.25}{\includegraphics{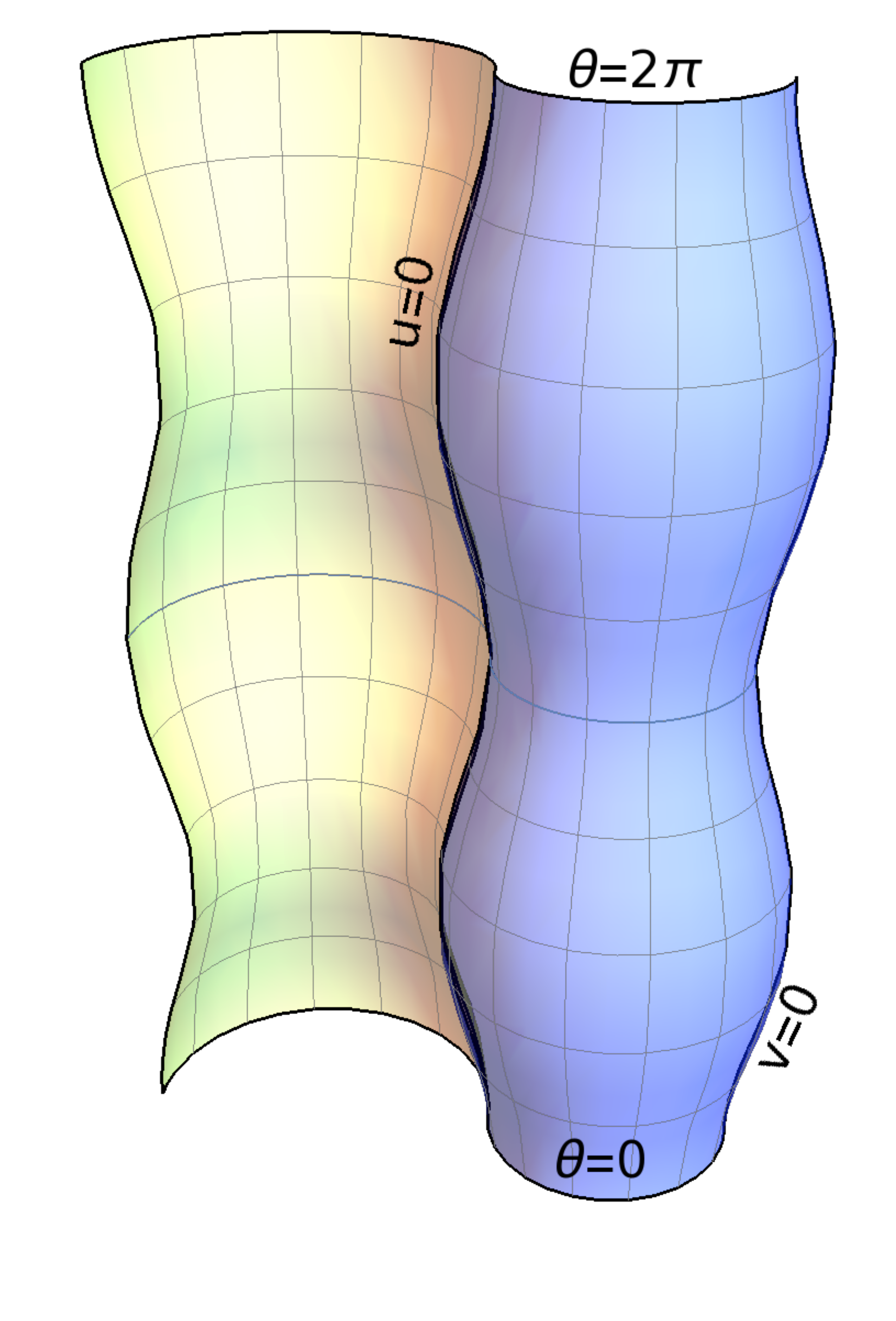}}
\vspace{-20pt}
\end{center}
\[\hbox{Figure 3: Surfaces of revolution associated 
to $\sF_2'$ and $\sF_3'$}\]

\vs\vs

Since $\sF_2'$ incorporates pairs $\buv$ in which the two vectors make
an arbitrary angle, we can use its geometry to measure the relative
orientation of the two $\R^3$ subspaces defined by $v=0$ and
$u=0$. The so-called `principal angles' between these subspaces are
determined by the function
\[\th\lmt\int_0^\pi\!\!{\ts\frac1{\sqrt2}}f(\th)\,d\phi = 
\pi\sqrt{\ts\frac38-\frac18\cos2\th}.\] Its value varies from $\pi/2$
to $\pi/\sqrt2\sim 127^\mathrm{o},$ and the bulges in the surface of
revolution reflect the fact that a semicircle of radius $R=1$ has
circumference $\sqrt2\pi$ when $\th=\pi/2$ or $3\pi/2$. The latter
corresponds to the situation of Proposition \ref{1ff} in which the two
vectors are anti-aligned.

Observe that $\pq^{-1}(\sF_i')$ is diffeomorphic to a cone over
$S^1\times S^2$ for $i=1,2,3$. Corollary \ref{van} tells us that the
connection $\Th_2$ is flat over $\R^4,$ so there it equals some exact
1-form $d\psi$. The restriction of the $\2$ metric $\hh_2$ to
$\pq^{-1}(\sF_2')$ equals
\[ \ba{rcl} \pq^*\gg_2+\qart N_c\,\Th_2\!^2
&=& \pq^*\gg_2+2uv(3-\cos2\th)d\psi^2\yy
&=& dR^2 + \frac12R^2\left[d\th^2+\qart(3-\cos2\th)
  \big(d\phi^2+\sin^2\!\phi\>d\psi^2\big)\right].\ea\]
Parallel semi-circles in the surface of revolution are the images by $\pq$ of
2-spheres with latitude $\psi,$ which {\gr collapse} at north and south poles
lying over the meridians $u=0$ and $v=0$. In this way, the circle quotient over
$\sF_2'$ is modelled metrically on the height function $S^2\to[-1,1]$.

\bibliographystyle{acm}
\bibliography{red}

\enddocument